\documentclass[letterpaper,10pt]{IEEEtran}

\IEEEoverridecommandlockouts

\usepackage{color}
\usepackage{enumerate,graphicx,epstopdf}
\usepackage{amsmath,amssymb,bm,amsthm}
\usepackage{cite}
\usepackage{mathtools}
\usepackage{array}
\usepackage{algpseudocode,algorithm,algorithmicx}
\usepackage{color}
\usepackage{booktabs}
\usepackage[hyphens]{url}
\usepackage{lipsum}
\usepackage[dvipsnames]{xcolor}
\usepackage[deletedmarkup=sout,authormarkup=superscript]{changes}
\usepackage[normalem]{ulem}
\usepackage{flushend}

\newtheorem{ass}{Assumption}
\newtheorem{rmk}{Remark}

\newtheorem{lmm}{Lemma}

\newtheorem{cor}{Corollary}
\newtheorem{thm}{Theorem}

\newcommand{\ints}{\mathbb{N}}
\newcommand\intrng[2]{\ints_{[#1,#2]}}
\newcommand{\reals}{\mathbb{R}}

\newcommand{\mc}[1]{\mathcal{#1}}

\newcommand{\proj}[1]{\Pi_{#1}}
\newcommand{\Int}[0]{\mathrm{Int}~}
\renewcommand{\ker}[0]{\mathrm{Ker}~}

\newcommand{\defeq}[0]{\equiv}
\DeclareMathOperator*{\argmin}{arg\,min}

\definechangesauthor[name={Marco}, color=NavyBlue]{MN}
\definechangesauthor[name={Dominic}, color=RedViolet]{DLM}
\definechangesauthor[name={Terry}, color=OliveGreen]{TS}
\definechangesauthor[name={Torbjørn}, color=Purple]{TC}

\begin{document}
\title{A Feasibility Governor for Enlarging the Region of Attraction of Linear Model Predictive Controllers \thanks{D. Liao-McPherson is with ETH Zürich. Email: \texttt{dliaomc@ethz.ch}.}
\thanks{T. Cunis and I. Kolmanovsky are with the University of Michigan, Ann Arbor. Email: \texttt{\{tcunis, ilya\}@umich.edu}.} \thanks{ M. Nicotra and T. Skibik are with the University of Colorado, Boulder, Email: \texttt{\{marco.nicotra, terrence.skibik\}@colorado.edu}.} \thanks{This research is supported by the National Science Foundation through awards CMMI 1904441 and CMMI 1904394 and by the Toyota Research Institute (TRI). TRI provided funds to assist the authors with their research but this article solely reflects the opinions and conclusions of its authors and not TRI or any other Toyota entity.}}
\author{Dominic Liao-McPherson, Terrence Skibik, Torbjørn Cunis, Ilya Kolmanovsky, and Marco M. Nicotra}

\maketitle

\begin{abstract}
This paper proposes a method for enlarging the region of attraction of Linear Model Predictive Controllers (MPC) when tracking piecewise-constant references in the presence of pointwise-in-time constraints. It consists of an add-on unit, the Feasibility Governor (FG), that manipulates the reference command so as to ensure that the optimal control problem that underlies the MPC feedback law remains feasible. Offline polyhedral projection algorithms based on multi-objective linear programming are employed to compute the set of feasible states and reference commands. Online, the action of the FG is computed by solving a convex quadratic program. The closed-loop system is shown to satisfy constraints, be asymptotically stable, exhibit zero-offset tracking, and display finite-time convergence of the reference.
\end{abstract}


\section{Introduction}
Model Predictive Control \cite{rawlings2018model,goodwin2006constrained,mayne2014model} (MPC) defines a feedback policy as the solution of a receding horizon optimal control problem (OCP). MPC is widely used in applications; it enables high-performance control while systematically enforcing state and control constraints and is supported by a robust theoretical literature. Stability guarantees are typically obtained by incorporating ``terminal ingredients'' into the OCP. For example, adding a terminal penalty and an invariant set based terminal constraint is sufficient to guarantee asymptotic stability and constraint satisfaction\cite{mayne2000mpc,chen1998quasi}; the closed-loop region of attraction (ROA) is then the set of all states from which it is possible to reach the terminal set within the prediction horizon. 

Many practical applications of MPC require the capability to track non-zero steady state references and to safely transition between them. However, if the change in the reference is large the system may not be able to reach the new terminal set within the prediction horizon, resulting in infeasibility and failure of the MPC controller. The obvious strategy for avoiding infeasibility is increasing the size of the ROA. This can be done by enlarging the terminal set or increasing the prediction horizon. Unfortunately, the maximum size of the terminal set is fixed by the constraints \cite{gilbert1991linear}, and increasing the prediction horizon increases the computational footprint of the controller. 

Another strategy is to treat aspects of the terminal set, e.g., size, location, or shape, as optimization variables and use these additional degrees of freedom to enlarge the feasible set. This approach has been applied to economic operation of nonlinear systems with terminal state constraints \cite{fagiano2013generalized} and regulation of linear systems using terminal set constraints \cite{gonzalez2009enlarging}. It has also been applied to reference tracking problems for linear systems \cite{limon2008mpc,simon2014reference} using various parameterizations of the terminal sets. Computing a contractive sequence of terminal sets offline which are incorporated into the OCP to enlarge the ROA is proposed in \cite{limon2005enlarging}. The major disadvantage of these approaches is that they require redesigning the OCP and increasing the computational complexity of the controller.

\begin{figure}[htbp]
	\centering
	\includegraphics[width=\columnwidth]{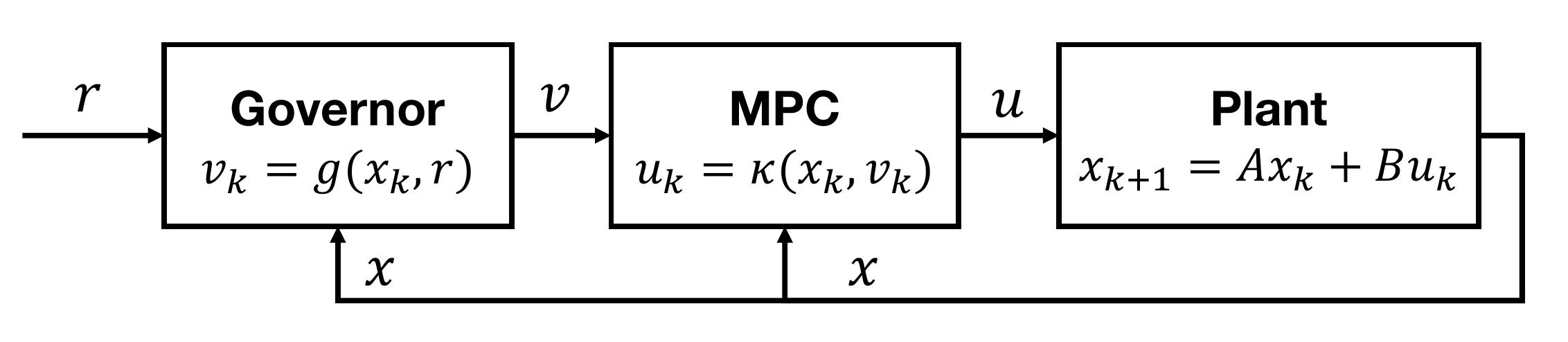}
	\caption{A block diagram of the control architecture. Given a reference $r$, the Feasibility Governor manipulates the auxiliary reference $v$ to ensure that the primary MPC controller is able to produce a valid control input $u$.}
	\label{fig:fg_block_diagram}
\end{figure}

In this paper, we propose the Feasibility Governor (FG), an add-on unit in the tradition of reference/command governors \cite{garone2017reference,bemporad1997nonlinear}, that modifies the reference signal to ensure that the terminal set remains reachable within the prediction horizon. The FG does not require any modifications to the existing MPC controller, exhibits finite time convergence to the desired reference, and expands the ROA of the MPC controller to all states that can reach the terminal set of any steady state admissible reference. It also takes advantage of offline polyhedral set manipulation tools \cite{herceg2013multi,lohne2017vector} to limit online complexity and minimize conservatism. The proposed control architecture is illustrated in Figure~\ref{fig:fg_block_diagram}.

There is existing literature on avoiding infeasibility in MPC using reference manipulation. The dual-mode controller in \cite{chisci2003dual} features a recovery mode that simultaneously computes a modified reference and control input. This approach converges in finite-time but is invasive and may reduce performance. An FG like algorithm is proposed in \cite{olaru2005compact} and is used as an intermediate design stage in the construction of a piecewise affine control law that combines a governor and explicit MPC controller into a single unit. This approach suffers from the well known complexity limitations of explicit MPC \cite{bemporad2002explicit} as the dimension of the state, prediction horizon, and number of constraints increases. This paper shows that the FG can be scaled to larger systems/longer horizons, and provides a more detailed treatment of both the theoretical properties of the governor, including using under-approximation of the feasible set, and the computation of the terminal and feasible sets. A governor-like algorithm using ellipsoidal terminal sets is proposed in \cite{mayne2016generalized} and can be considered a special case of the FG that uses a specific reference parameterization and conservative inner approximation of the feasible set. In \cite{DEMPC_tac} the authors propose a suboptimal continuous-time analog of the governor in \cite{mayne2016generalized}. Finally, a spatial governor is proposed in \cite{di2018cascaded}. It is specific to precision machining applications and adjusts the velocity profile passed to a path tracking MPC controller to ensure recursive feasibility of constraints representing manufacturing error tolerances.

The layout of the paper is as follows: Section~\ref{ss:problem_setting} contains the problem setting and control objectives. Section~\ref{ss:control_strategy} describes the primary MPC controller. Section~\ref{ss:fg_design} introduces the FG, including implementation details, and summarizes its theoretical properties which are then proven rigorously in Section~\ref{ss:theoretical_properties}. Finally, Section~\ref{ss:numerical_examples} illustrates the utility of the FG through simulation studies and Section~\ref{ss:conclusions} offers some conclusions and perspectives on future work.

\subsection{Notation}
For vectors $a$ and $b$, $(a,b) = [a^T~~b^T]^T$. The identity and zero matrices are denoted $I_N \in \reals^{N\times N}$ and $0_{N \times M} \in \reals^{N \times M}$, respectively with the subscripts absent whenever the dimensions are clear from context. Given $M\in \reals^{m\times n}$ and $\mc{U}\subseteq \reals^n$, $\ker M = \{x~|~Mx = 0\}$, $M \mc{U} = \{Mx~|~x\in \mc{U}\}$, $M^{-1} \mc{U} = \{x~|~Mx \in \mc{U}\}$, and $\Int\!\mc{U}$ denotes the interior of $\mc{U}$. Set addition/subtraction is defined as $\mc{U} \pm \mc{V} = \{u\pm v~|~(u,v) \in \mc{U} \times \mc{V}\}$ and for $\lambda \in \reals$, $\lambda \mc{U} = \{\lambda u~|~u\in \mc{U}\}$. Positive (semi) definiteness of a matrix $P \in \reals^{n \times n}$ is denoted by $(P \succeq 0)$ $P \succ 0$; and $\|x\|_P = \sqrt{x^TPx}$ for $x \in \reals^n$. Consider $x\in \reals^n$, $y\in \reals^m$ and a set $\Gamma \subseteq \reals^{n+m}$, the projection of $\Gamma$ onto $x$ is the image $\Pi_x \Gamma$ where $\Pi_x = [I_{n}~0_{n\times m}]$, i.e., $x = \Pi_x [x^T,y^T]^T$. The slice (or cross-section) operation is $S_y(\Gamma,x) = \{y~|~(x,y)\in \Gamma\}$. For $x\in \reals^n$, and $\delta \geq 0$, $\mc{B}_\delta(x) = \{y~|~\|y-x\|\leq \delta \}$ For a sequence $\{x_k\}\subseteq \reals^n$ and a set $\Gamma \subseteq \reals^n$ we write that $x_k \to \Gamma$ as $k\to \infty$, if and only if $\lim_{k\to\infty}~\inf_{y\in \Gamma} \|y-x_k\| = 0$. Our use of comparison functions, i.e., class $\mc{K},\mc{K}_\infty$ and $\mc{KL}$ functions follows \cite{kellett2014compendium}.

\section{Problem Setting} \label{ss:problem_setting}
Consider the linear time invariant (LTI) system 
\begin{subequations} \label{eq:LTI_system}
\begin{align} 
x_{k+1} &= A x_k + B u_k\\
y_k &= Cx_k + D u_k\\
z_k &= E x_k + F u_k,\end{align}
\end{subequations}
where $k\in \ints$ is the discrete-time index and $x_k \in \reals^{n_x}$, $u_k \in \reals^{n_u}$, $y_k \in \reals^{n_y}$, and $z_k \in \reals^{n_z}$ are the states, control inputs, constrained outputs, and tracking outputs, respectively.
\begin{ass} \label{ass:stabilizable}
The pair $(A,B)$ is stabilizable.
\end{ass}

The system \eqref{eq:LTI_system} is subject to pointwise-in-time constraints
\begin{equation}
   \forall k \in \ints \quad y_k \in \mc{Y},
\end{equation}
where $\mc{Y}\subseteq \reals^{n_y}$ is a specified set of constraint.
\begin{ass}\label{ass:constraint_set}
The set $\mc{Y}$ is a compact polyhedron with representation $\mc{Y} = \{y~|~Y y \leq h\}$ and satisfies $0 \in \Int \mc{Y}$.
\end{ass}
\noindent As detailed in \cite{limon2008mpc}, Assumption~\ref{ass:stabilizable} implies that the matrix
\begin{equation}
    Z = \begin{bmatrix}
		I-A & B & 0\\
		E & F & -I
	\end{bmatrix}
\end{equation}
satisfies $\ker\!\!(Z)\neq\{0\}$. As a result, is possible to introduce an auxiliary reference $v\in \reals^{n_v}$ that parameterizes the equilibrium manifold, i.e., every solution to $Z~[x^T,u^T,z^T]^T = 0$, as
\begin{equation} \label{eq:xbardef}
\begin{bmatrix}
	\bar x_v \\ \bar u_v \\ \bar z_v
\end{bmatrix} = \begin{bmatrix}
	G_x\\G_u\\G_z
\end{bmatrix} v
\end{equation}
where $G^T \defeq \left[G_x^T~~G_u^T~~G_z^T\right]$ is a basis for $\ker\!\!(Z)$. 

The following assumption excludes pathological cases, e.g., $G_z = 0$, that are indicative of an ill-posed problem.
\begin{ass} \label{eq:G_full_rank}
The matrix $G_z$ is full rank.
\end{ass}
\begin{rmk}
The vector $v$ is a minimal parameterization of the equilibrium manifold of \eqref{eq:LTI_system}. If $G_z$ is not full row rank, e.g. if $n_z > n_v$, the output tracking problem is ill-posed and only $r\in G_z \reals^{n_v}$ are achievable. If $n_z = n_v$ and $G_z$ is invertible, the reference uniquely determines the target equilibrium and it is possible to choose $G$ such that $r = v$. If $n_z < n_v$, there are multiple equilibria satisfying $E\bar x_v+F\bar u_v = r$.
\end{rmk}
Next, we introduce a design parameter $\epsilon \in (0, 1)$ and the corresponding set of strictly steady-state admissible auxiliary references
\begin{equation} \label{eq:Veps}
   \mc{V}_\epsilon \defeq G_y^{-1} (1-\epsilon) \mc{Y} = \{v~|~
	 G_y v \in (1-\epsilon)\mc{Y}\},
\end{equation}
where $G_y = CG_x + DG_u$, and strictly admissible references 
\begin{equation}
	\mc{R}_\epsilon \defeq G_z \mc{V}_\epsilon = \{G_zv~|~v\in \mc{V}_\epsilon\}.
\end{equation} 
\begin{rmk}
The parameter $\epsilon$ is used because MPC controllers cannot stabilize points on the boundary of the feasible set.
\end{rmk}
\noindent Given Assumptions~\ref{ass:stabilizable}--\ref{ass:constraint_set} as the only limitations to our problem setting, we now state the control objectives of this paper. \medskip

\noindent\textbf{Control Objectives:} Given the LTI system \eqref{eq:LTI_system}, let $\mathcal{Y}\subseteq\reals^{n_y}$ be a set of constraints, and let $r\in\reals^{n_z}$ be a target reference. The goal of this paper is to design a full state feedback law that achieves the following objectives:
\begin{itemize}
    \item \textit{Safety:} Ensure $y_k\in\mathcal{Y}\quad \forall k \geq 0$;
    \item \textit{Convergence:} $\lim_{k\to\infty}z_k= r^*$, where 
    \begin{equation*}
      r^\star = \argmin_{s\in \mc{R}_\epsilon}~\|s-r\|.
    \end{equation*}
	\item \textit{Asymptotic Stability:} $\lim_{k\to\infty} (x_k,v_k)  = (x^*_r,v^*_r)$ where $(x^*_r,v^*_r) = (G_x v^*_r,v^*_r)$ is a stable equilibrium satisfying $r^* = G_z v^*_r$.
\end{itemize}

\begin{rmk}
When the tracking problem is well posed, i.e., $r\in \mc{R}_\epsilon$, we recover $\lim_{k\to\infty}z_k=r$.
\end{rmk}

\begin{rmk}
Assumption \ref{ass:constraint_set} restricts our setting to polyhedral constraints which simplifies some implementation aspects. All the theoretical results in this paper still hold under the weaker assumption that $\mc{Y}$ is compact, convex and contains the origin in its interior.
\end{rmk}

\section{Control Strategy} \label{ss:control_strategy}
Due to the constraints, we approach the control objectives using a typical MPC formulation where the feedback policy is defined using the solution to the following optimal control problem (OCP)
\begin{subequations} \label{eq:LMPC_OCP}
\begin{alignat}{2}
\underset{\mu}{\mathrm{min}}& &&||\xi_N - \bar x_v||_P^2 + \sum_{i=0}^{N-1} ||\xi_i - \bar x_v||_Q^2 + ||\mu_i - \bar u_v||_R^2 \label{eq:ocp_cost}\\
\mathrm{s.t.}& ~ &&~\xi_0 = x, \label{eq:ocp_cstr1} \\
& &&~\xi_{i+1} = A\xi_i + B \mu_i, ~~~ i \in \intrng{0}{N-1},\\
& &&~C \xi_i + D \mu_i \in \mc{Y},\label{eq:ocp_cstr2} ~~~~~~ i \in \intrng{0}{N-1},\\
& &&\qquad(\xi_N,v) \in \mc{T}, \label{eq:ocp_cstr3}
\end{alignat}
\end{subequations}
where $N\in \ints_{> 0}$ is the prediction horizon, $\mu = (\mu_0,\ldots \mu_{N-1})$ are the decision variables, $P$, $Q$, and $R$ are weighting matrices, and $\mc{T}\subseteq\reals^{n_x}\times\reals^{n_v}$ is the terminal set, which is assumed to be polyhedral\footnote{Other representations, e.g., ellipsoidal, are admissible but more challenging from an implementation perspective.}, i.e.,
\begin{equation}
	\mc{T} = \{(x,v)~|~ T_x x + T_v v \leq c\},
\end{equation}
and $\bar x_v, \bar u_v$ are defined in \eqref{eq:xbardef} and will be manipulated.

\begin{rmk}
The terminal constraint \eqref{eq:ocp_cstr3} is often written in the equivalent form $\xi_N\!\in\! \mc{X}(v)= S_x(\mc{T},v) = \{x~|~(x,v) \in \mc{T}\}$.
\end{rmk}

The following assumptions ensure that \eqref{eq:LMPC_OCP} is well-posed and can be used to construct a stabilizing feedback law.

\begin{ass}\label{ass:DARE}
The stage cost matrices satisfy $Q = Q^T \succeq 0$, with $(A,Q)$ observable, and $R = R^T \succ 0$.
\end{ass}

Once the stage weights are defined, the terminal penalty $P$ and the terminal set mapping $\mc{X}$ can be obtained using a gain $K \in \reals^{n_u \times n_x}$ and a fictitious terminal control law
\begin{equation} \label{eq:terminal_control_law}
	\kappa_N(x,v) \defeq \bar u_v -K(x-\bar x_v).
\end{equation}

\begin{ass} \label{ass:terminal_cost}
Given $(x,v)\in \mc{T}$ and the terminal control law \eqref{eq:terminal_control_law}, the terminal cost matrix satisfies $P = P^T\succeq0$ and
\begin{equation}\label{eq:lyapunov}
	\|(A-BK)\delta x\|_P^2 - \|\delta x\|_P^2 \leq -\|\delta x\|_{(Q+K^TRK)}^2
\end{equation}
where $\delta x = x - \bar x_v$.
\end{ass}

\begin{ass} \label{ass:terminal_set}
The terminal set $\mc{T}$ is invariant and constraint admissible under \eqref{eq:terminal_control_law}, i.e., for all $(x,v)\in\mathcal{T},$
\begin{subequations}\label{eq:terminal_set}
\begin{align}
    Ax + B\kappa_N(x,v)&\in\mathcal{X}(v), \\
    Cx + D\kappa_N(x,v)&\in\mathcal{Y}.
\end{align}
\end{subequations}
\end{ass}

\begin{rmk}
The terminal control law \eqref{eq:terminal_control_law} is not used online but is needed to synthesize $P$ and $\mc{T}$. A conservative choice is $K = 0$, $P = 0$, and $\mc{T} = \{(\bar x_v,v)\}$. Alternatively, for any $K$ such that $A-BK$ is Schur, $P$ can be obtained by reformulating and then solving \eqref{eq:lyapunov} as a discrete Lyapunov equation\footnote{Given the linear quadratic regulator $K = (R+B^TPB)^{-1}(B^T PA)$, the discrete Lyapunov equation \eqref{eq:lyapunov} coincides with the discrete Riccati equation $P = Q+A^TPA-(A^TPB)(R+B^TPB)^{-1}(B^T PA)$.}. Polyhedral approximations of the largest possible set $\mc{T}$ can then be computed offline as detailed in Appendix A.
\end{rmk}

It is only possible to compute a control action if \eqref{eq:LMPC_OCP} admits a solution. The set of all parameters for which this is possible, i.e., the feasible set, is
\begin{equation}
    \label{eq:feasible_set}
	\Gamma_N \defeq \{(x,v)~|~\exists~\mu:~ \eqref{eq:ocp_cstr1}-\eqref{eq:ocp_cstr3}\} \subseteq \reals^{n_x}\times\reals^{n_v},
\end{equation}  
which is the $N$-step backwards reachable set of $\mc{T}$. The set of strictly steady-state admissible equilibria is
\begin{equation}\label{eq:Sigma}
   \Sigma \defeq \{(x,v)~|~
	 x=G_x v,~ v \in\mc{V}_\epsilon \} .
\end{equation} 

If $\mc{T}$ is polyhedral, then $\Gamma_N$ is polyhedral as well and can be computed offline, see Section~\ref{ss:feasible_set_computation}. Figure~\ref{fig:Lemma1(1)} illustrates the various sets defined in this section.
\begin{figure}[h]
	\centering
	\includegraphics[width = \columnwidth]{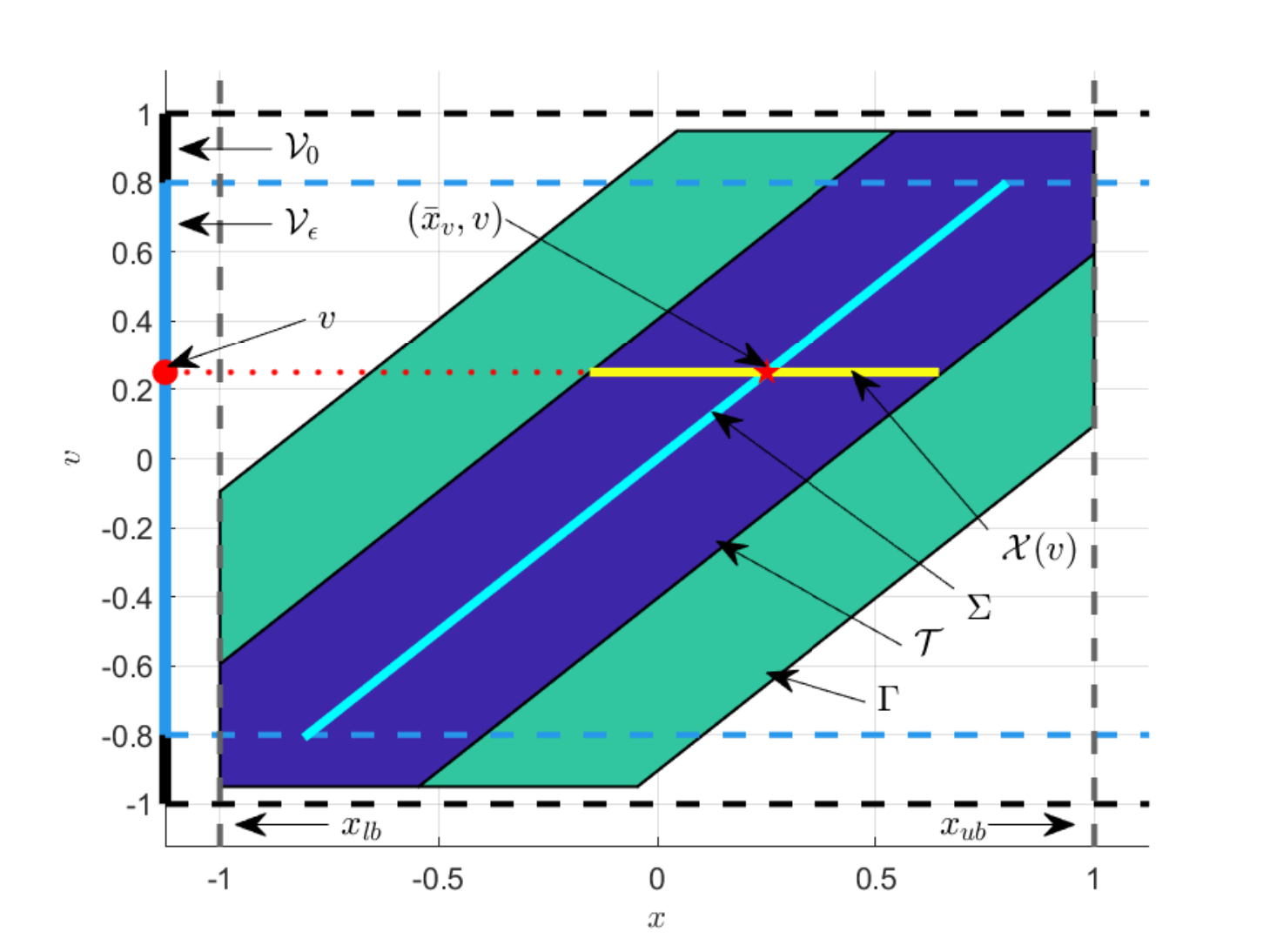}
	\caption{The sets used in the paper for the integrator $x_{k+1} = x_k + u_k$ subject to $|x_k| \leq 1$, $|u_k|\leq 0.25$ and with $\epsilon = 0.2$, $\mathcal{T} = \tilde{O}_{\infty}^{0.05}$ and $N=2$}
	\label{fig:Lemma1(1)}
\end{figure}

The following technical assumption is needed to guarantee convergence and always holds when $\mc{T}$ is synthesized using the procedure in Appendix A.

\begin{ass} \label{ass:xv_in_interior}
$\Sigma \subset \Int\Gamma_N$
\end{ass}

\begin{lmm}\label{lmm:strictly_ss}
Given Assumption~\ref{ass:stabilizable} either of the following conditions are sufficient for Assumption~\ref{ass:xv_in_interior} to hold.
\begin{enumerate}
    \item $\Sigma \subset \Int \mc{T}$;
    \item $(A,B)$ is controllable, $\Sigma \subseteq \mc{T}$, and $N \geq \nu$, where the controllability index~$\nu$ is the smallest positive integer such that $\begin{bmatrix}
	B & AB & \cdots & A^{\nu-1}B
\end{bmatrix}$ is full rank.
\end{enumerate}
\end{lmm}
\begin{proof}
See Appendix C.
\end{proof}

The MPC feedback policy $\kappa : \Gamma_n \to \reals^{n_u}$ is
\begin{equation} \label{eq:MPC_feedback}
	\kappa(x,v) \defeq \mu^\star_0(x,v)
\end{equation}
where $\mu^\star(x,v)=[\mu_0^{\star T},\mu_1^{\star T},\ldots,\mu_{N-1}^{\star T}]^T$ is the minimizer of \eqref{eq:LMPC_OCP}, and is defined for $(x,v)\in \Gamma_N$. The following theorem summarizes the properties of the closed-loop system for a constant auxiliary reference.
\begin{thm} \label{thm:MPC_stab}
Let Assumptions~\ref{ass:stabilizable}--\ref{ass:terminal_set} hold and let $\phi(\ell,x,v)$ denote the solution of the closed-loop dynamics
\begin{equation} \label{eq:closed_loop_dynamics}
x_{k+1} = f(x_k,v) \equiv Ax_k + B\kappa(x_k,v).
\end{equation} 
starting from the initial condition $x_0 = x$ at timestep $\ell \geq 0$.
Then for all $(x,v) \in \Gamma_N$:
\begin{itemize}
    \item $(\phi(\ell,x,v),v)\in\Gamma_N,~\forall \ell\geq 0$;
    \item $y_\ell \in \mc{Y},~\forall \ell\geq 0$;
    \item $\lim_{\ell\to\infty} \phi(\ell,x,v) = \bar x_{v}$.
\end{itemize}
If, in addition, $v \in \Int \mc{V}_0$ then $\bar x_{v}$ is asymptotically stable.
\end{thm}
\begin{proof}
Since the auxiliary reference $v$ is constant for  $\ell\geq 0$, the statement follows from \cite[Theorem 4.4.2]{goodwin2006constrained}.
\end{proof}

Theorem \ref{thm:MPC_stab} achieves the control objectives given $v_0$ such that $G_zv_0=r$ and $x_0$ satisfying $(x_0,v_0)\in\Gamma_N$. Its main limitation, however, lies in the fact that the OCP \eqref{eq:LMPC_OCP} is infeasible if $x_0$ cannot be steered to $\mathcal{X}(v_0)$ within $N$ steps. Although increasing the prediction horizon may seem like a suitable workaround, this solution may be inapplicable in practice since the computational time required to solve \eqref{eq:LMPC_OCP} scales unfavorably with $N$. 

In the next section, we describe an \emph{add-on} unit that expands the closed-loop domain of attraction without extending the prediction horizon or modifying the MPC formulation.

\section{The Feasibility Governor}  \label{ss:fg_design}
The MPC feedback policy \eqref{eq:MPC_feedback} is stabilizing only if the terminal set associated with the target equilibrium is $N$-step reachable from the current state. Intuitively, if the target can be manipulated, this limitation can be overcome by selecting a sequence of intermediate targets that are pair-wise reachable. This paper formalizes this idea by redefining the auxiliary reference $v$ as a time-varying signal $v_k$ to ensure $(x_k,v_k)\in\Gamma_N,~\forall k\in\ints$ and $G_zv_{k}=r$ for sufficiently large $k\in\ints$. The resulting control architecture is displayed in Figure~\ref{fig:fg_block_diagram}.

\subsection{Governor Design}
The idea behind the FG is straightfoward, minimally modify the reference so that the MPC problem remains feasible. Drawing inspiration from the command governor (CG) literature \cite{bemporad1997nonlinear}, the action of the FG can be computed via the following optimization problem.
\begin{subequations} \label{eq:FG_1}
\begin{alignat}{2}
	\min_{v\in \mc{V}_\epsilon}& \quad && \|G_z v - r\|_2^2 \label{eq:FG1_cost} \\
	\mathrm{s.t.}~&&& (x,v) \in \Gamma_N.
\end{alignat}
\end{subequations}
At time $k$, given a measurement $x_k$, the FG computes a virtual reference $v_k$ as a solution to \eqref{eq:FG_1} with $x = x_k$ that is passed to the MPC controller to obtain a control action $u_k = \kappa(x_k,v_k)$. 

The FG can be considered an extension of the CG and operates on the same principle: manipulate the auxiliary reference to remain within a safe invariant set associated with an underlying primary controller. In the case of the CG, the invariant sets are typically slices of $O_\infty$, the maximum constraint admissible set \cite{gilbert1991linear} associated with a linear feedback law such as \eqref{eq:terminal_control_law}. In contrast, the FG uses slices of $\Gamma_N$ which are invariant under the nonlinear MPC feedback \eqref{eq:MPC_feedback}. Assuming the common choice $\mc{T} = O_\infty$, the set $\Gamma_N$ is a superset of $O_\infty$ and grows larger as $N$ increases, as illustrated in Figures~\ref{fig:gamma_vs_T} and \ref{fig:DOA_vs_N}. The use of a more permissive constraint set leads to better performance, as the MPC controller is ``aware'' of the constraints, which is not possible using linear feedback.

Unfortunately, if $G_z$ does not have full column rank then $\|G_z v - r\|_2^2$ is not strongly convex and \eqref{eq:FG_1} will not have a unique minimizer. This is problematic from a convergence perspective; a mechanism for resolving degeneracies is needed. As such, we extend \eqref{eq:FG_1} and define the FG feedback as
\begin{equation} \label{eq:fg_opt_problem}
	g(x,r) \defeq \argmin_{v\in \mc{V}_\epsilon}~\left \{\psi(v,r)~|~(x,v) \in \Gamma_N\right\}
\end{equation}
where
\begin{equation} \label{eq:psi_def}
	\psi(v,r) \defeq \begin{cases}
	\|G_z v - r\|_2^2 & \text{if $G_z$ is injective} \\
	\|v - v^*_r\|_2^2 & \text{otherwise},
	\end{cases}
\end{equation}
and the designer can select any $v^*_r$ satisfying
\begin{equation} \label{eq:vstar}
	v^*_r \in \mc{V}^*_r \defeq \argmin_{v \in \mc{V}_\epsilon}~\|G_z v - r\|_2^2.
\end{equation}
Note that $\psi$ defined in this way is strongly convex is $v$.

The resulting feedback law is $u_k = \kappa(x_k,g(x_k,r))$ and the closed-loop system dynamics are
\begin{equation}
	x_{k+1} = Ax_k + B\kappa(x_k,g(x_k,r)).
\end{equation} 
Since the function $\psi$ is strongly convex, \eqref{eq:fg_opt_problem} is a convex quadratic program (with a unique solution) that can be solved reliably online.
\begin{rmk}
If the reference is achievable, i.e., $r\in \mc{R}_\epsilon$, then $G_z^\dagger r \in \mc{V}_r^*$, where $G_z^\dagger$ is the Moore-Penrose pseudo-inverse.
\end{rmk}

\subsection{Properties}

When combined with \eqref{eq:MPC_feedback} and placed in closed-loop with \eqref{eq:LTI_system}, the combined FG + MPC feedback policy ensures constraint satisfaction, renders the point $(x^*_r,v^*_r) = (G_x v^*_r,v^*_r)$ asymptotically stable, and exhibits finite time convergence of $v_k \to v^*_r$. These results are rigorously formulated and proven in Section V.

Moreover, the addition of the FG expands the domain of attraction of the closed-loop system from $\mc{D}_{MPC} = S_x(\Gamma_N,v^*_r)$, the set of states from which it is possible to reach $\mc{X}(v^*_r)$ in $N$-steps, to
\begin{equation}
	\mc{D}_{FG} = \bigcup_{v\in \mc{V}_\epsilon} S_x(\Gamma_N,v),
\end{equation}
the set of states from which it is possible to reach $\mc{X}(v)$ of \textit{any} $v\in \mc{V}_\epsilon$ in $N$-steps. In particular, the addition of the FG guarantees safe transitions between any $r_1,r_2 \in \mc{R}_\epsilon$. The differences between $\mc{D}_{MPC}$ and $\mc{D}_{FG}$ are illustrated in Figure~\ref{fig:2D_traj} for the double integrator example in Section~\ref{ss:double_integrator}.

\begin{rmk}
The FG can be applied to systems with disturbances by noting that the essential property required by the FG is that the feasible set $\Gamma_N$ of the model predictive controller is forward invariant for any constant auxiliary reference. We can readily replace the MPC formulation \eqref{eq:LMPC_OCP} with any alternative OCP with a forward invariant feasible set. For example, the tube MPC formulation \cite[Algorithm 3.1]{kouvaritakis2016model}, based on the theory of Robustly Positive Invariant Sets \cite{kolmanovsky1998theory}, would be a valid choice as it renders its feasible set disturbance invariant \cite[Theorem 3.2]{kouvaritakis2016model}.
\end{rmk}

\begin{figure}[h]
	\centering
	\includegraphics[width = \columnwidth]{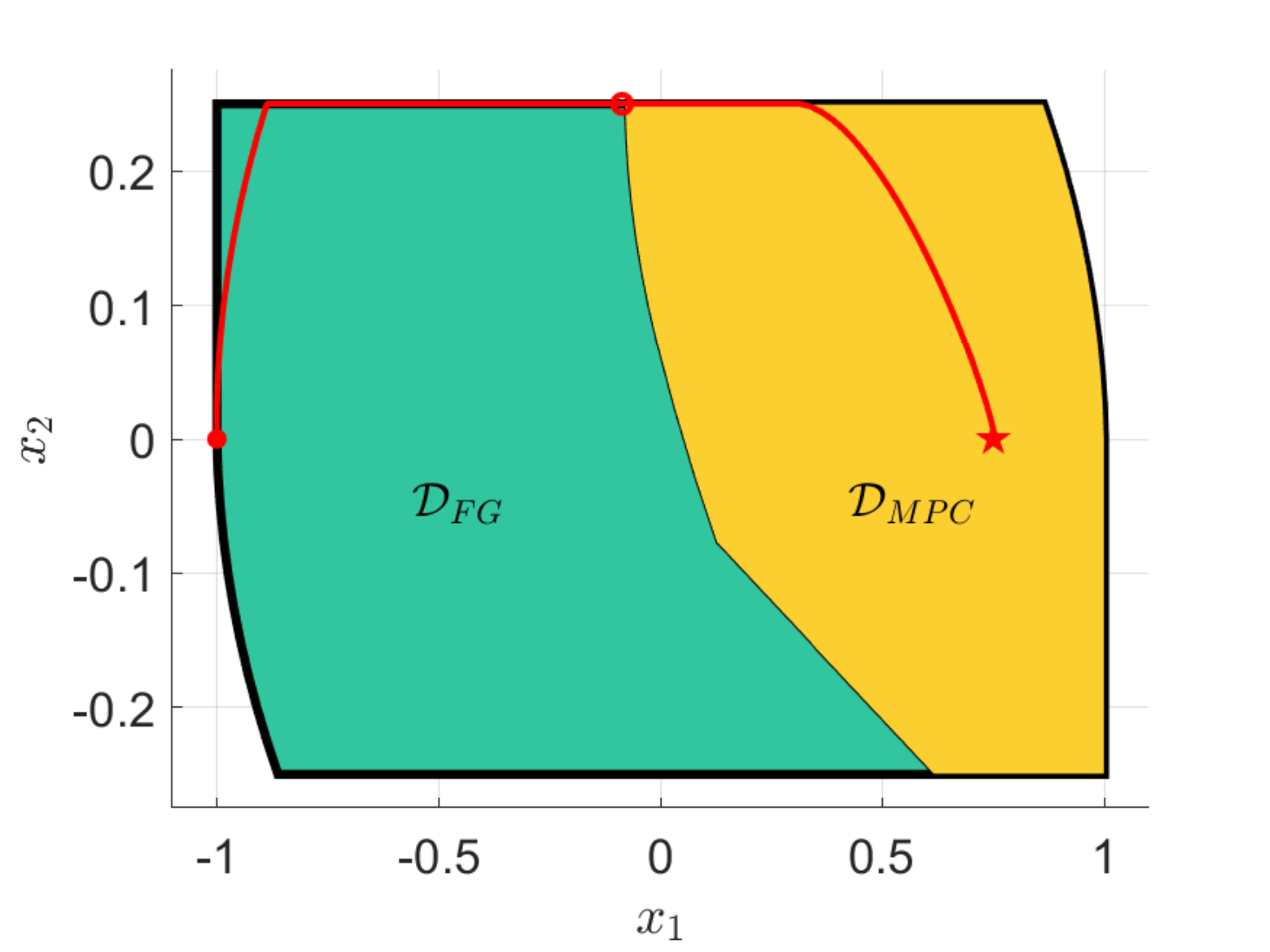}
	\caption{For the double integrator example in Section~\ref{ss:double_integrator}, the region of attraction of the combined MPC + FG feedback law (green) is larger than that of the MPC controller alone (yellow).}
	\label{fig:2D_traj}
\end{figure}

\subsection{Implementation} \label{ss:feasible_set_computation}
In our problem setting, $\Gamma_N$ and $\mc{V}_\epsilon$ are polyhedral and thus \eqref{eq:fg_opt_problem} is a strongly convex quadratic program (QP). Convex QPs can be solved efficiently and reliably using active set, interior point, proximal gradient, or generalized Newton methods. The problem \eqref{eq:fg_opt_problem} typically has only a small number of variables and many constraints. For example, the lateral vehicle example in Section~\ref{ss:lateral_vehicle} has $1$ variable and around $6000$ inequality constraints. Dual active-set methods \cite{goldfarb1983numerically} can solve the FG problems efficiently and reliably; they start from the unconstrained optimum and only consider a limited number of active constraints at a time.

Implementation of the FG also requires a half-space representation of the feasible set. Two methods for obtaining one via polyhedral calculus are described below. 
\subsubsection{Block Method}
The MPC OCP \eqref{eq:LMPC_OCP} is a QP and can be written in the condensed form\footnote{Expressions for the matrices in \eqref{eq:Condensed_LMPC_OCP} are provided in Appendix B.}
\begin{subequations} \label{eq:Condensed_LMPC_OCP}
\begin{alignat}{2} 
\underset{\mu}{\mathrm{min.}}&\quad \frac12&& \mu^T H \mu + \mu^T W\theta\\
\mathrm{s.t.}& && M \mu + L \theta \leq b,
\end{alignat}
\end{subequations}
with parameter $\theta = (x,v)$. The feasible set \eqref{eq:feasible_set} can therefore be expressed as
\begin{equation} \label{eq:gamma_block}
	\Gamma_N = \proj{\theta} \{(\mu,\theta)~|~ M \mu + L \theta \leq b\}.
\end{equation}

\subsubsection{Recursive Method}
The feasible set $\Gamma_N$ is the $N$-step backwards reachable set of $\mc{T}$ and can be computed recursively. Define the matrices
\begin{equation}
	A_e = \begin{bmatrix}
	A & 0\\
	0 & I
	\end{bmatrix}~~B_e = \begin{bmatrix}
		B \\0
	\end{bmatrix}, \text{ and } M_e = \begin{bmatrix}
		A_e & B_e
	\end{bmatrix},
\end{equation}
and the set $\mc{W} = \{(x,v,u)~|~C x + Du \in \mc{Y}\}$.
Then $\Gamma_N$ can be computed via the recursion
\begin{equation} \label{eq:gamma_recursive}
	\Gamma_{i+1} = \proj{\theta}\left(M_e^{-1}\Gamma_i \cap \mc{W}\right),
\end{equation}
starting from the initial condition $\Gamma_0 = \mc{T}$. \\

There are several toolboxes available for performing polyhedral calculus (e.g., projections, images, inverse images etc.). In this paper we use and compare the MPT3 \cite{herceg2013multi} and \texttt{bensolve tools} \cite{ciripoi2018calculus} packages.

For both the recursive and block methods, the complexity of computing $\Gamma_N$ is dominated by the projection operation. The projection is performed offline but can quickly become intractable, even for moderately sized systems, as all known projection algorithms suffer from the curse of dimensionality \cite{huynh1992practical}. Thus computing $\Gamma_N$ can quickly become intractable as the size of the state vector, input vector, reference, or prediction horizon grows. 

In this paper, we investigate two different methods for computing the projections, multi-objective linear programming (MOLP) \cite{lohne2016equivalence}, implemented in \texttt{bensolve tools} \cite{ciripoi2018calculus}, and Fourier-Motzkin elimination, implemented in the MPT3 toolbox \cite{herceg2013multi}. We computed $\Gamma_N$ for several values of $N$ for the double integrator (Section~\ref{ss:double_integrator}) and lateral vehicle model (Section~\ref{ss:lateral_vehicle}) and recorded the execution (wall-clock) time on a 2019 Macbook Pro (2.8 GHz i9, 32GB RAM) running MATLAB 2019b. We observe that the MOLP method significantly outperforms Fourier elimination in term of both speed and reliability, as seen in Figure~\ref{fig:di_scaling}. Further, the recursive method is marginally faster than the block method, as illustrated in Figures~\ref{fig:di_scaling} and \ref{fig:lv_scaling}. Both methods display exponential scaling in the horizon length, as expected for polyhedral projection methods.

\begin{figure}[htbp]
	\centering
	\includegraphics[width=0.95\columnwidth]{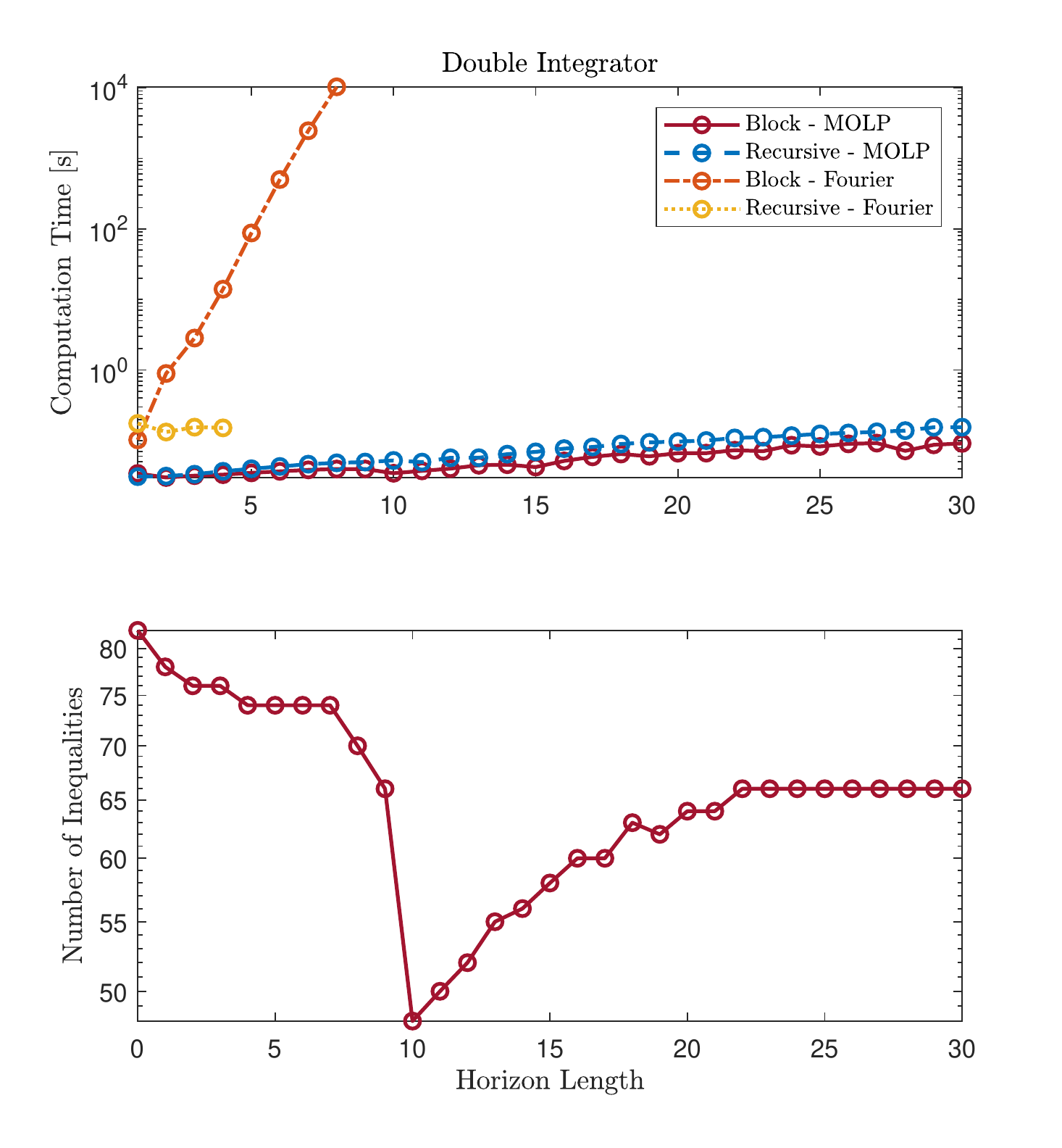}
	\caption{The computational cost of computing $\Gamma_N$. The MOLP projection algorithm outperforms Fourier elimination (0.16 seconds for $N = 30$ vs. 2.9 hours for $N = 8$) and the Fourier-recursive method fails for $N\geq 5$. The number of inequalities necessary to represent $\Gamma_N$ eventually converges, we hypothesize because $\Gamma_N$ must be contained in the state constraints which are a simple box.}
	\label{fig:di_scaling}
\end{figure}

\begin{figure}[htbp]
	\centering
	\includegraphics[width=0.95\columnwidth]{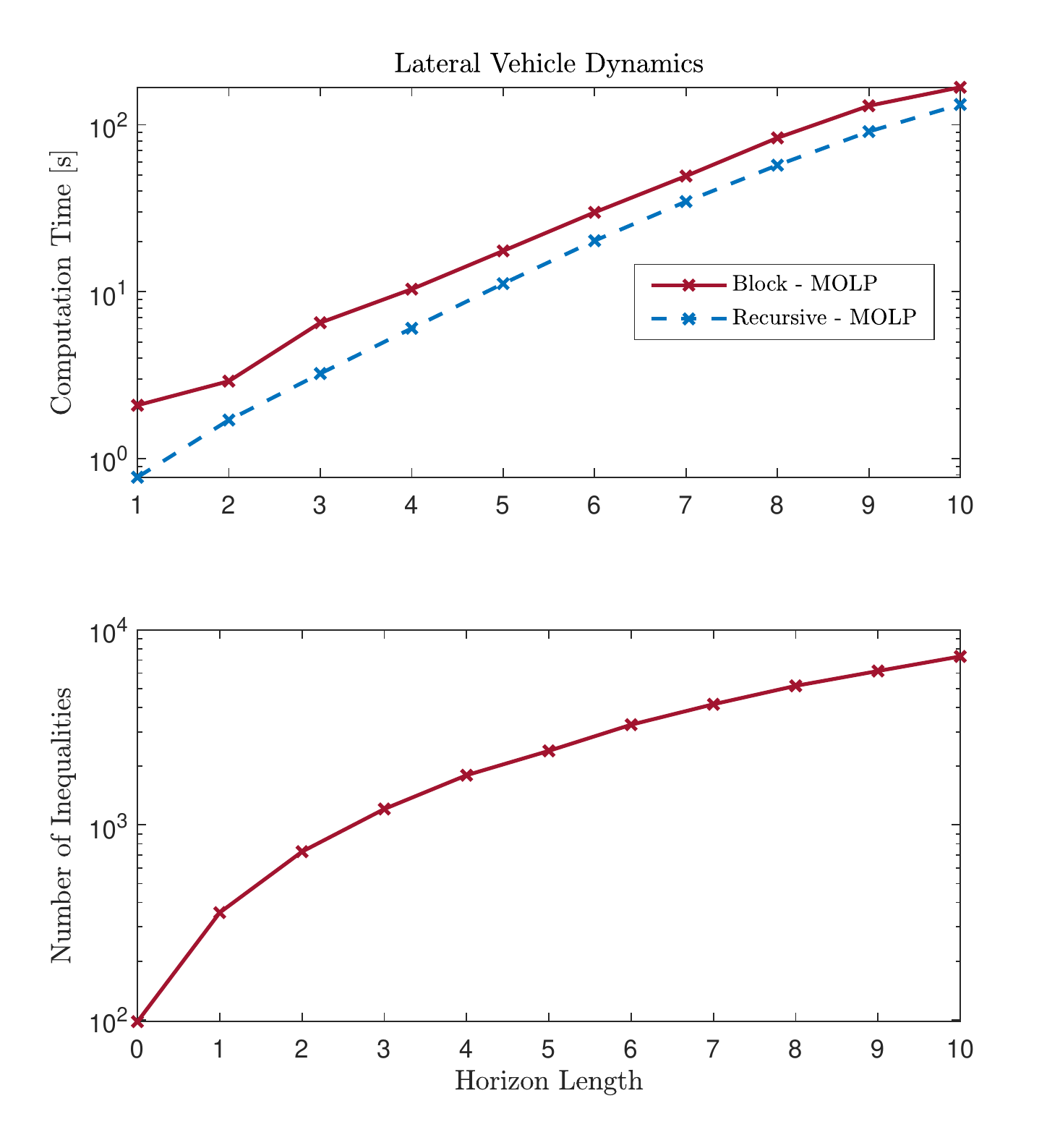}
	\caption{The computational cost of computing $\Gamma_N$. The recursive method marginally outperforms the block method with both displaying slow exponential growth.}
	\label{fig:lv_scaling}
\end{figure}

Our investigations confirm that projection based methods for computing $\Gamma_N$ are tractable only for moderately sized systems. One strategy for applying the FG to larger systems is to replace $\Gamma_N$ with an easier to compute approximation.

\subsection{Under-approximating the Feasible Set} \label{ss:under-approx}
In some scenarios, it may be advantageous (or necessary) to use an approximation of the feasible set. Luckily, with some minor modifications, a set $\mc{F} \subseteq \Gamma_N$ can be used in place of $\Gamma_N$. In this case, the FG is re-defined as follows
\begin{subequations}  \label{eq:fg_under_approx}
\begin{equation}
 	g(x,v,r) \defeq  \begin{cases}
 		\bar{g}(x,r) & \text{if } (x,v) \in \mc{F}\\
 		v & \text{else}
 	\end{cases}
 \end{equation} 
 where 
 \begin{equation}
 	\bar{g}(x,r) \defeq \argmin_{v\in \mc{V}_\epsilon}~\left\{\psi(v,r)~|~(x,v)\in \mc{F}\right\}.
 \end{equation}
 \end{subequations}
At time $k$ the auxiliary reference is then  computed as $v_{k} = g(x_k,v_{k-1},r)$.
The set $\mc{F}$ must satisfy the following:
\begin{ass}\label{ass:under_approx}
The set $\mc{F}\subseteq\Gamma_N$  is closed, convex, polyhedral, and satisfies $\Sigma \subset \Int \mc{F}$. 
\end{ass}

The idea behind \eqref{eq:fg_under_approx} is that, while the slices of $\mc{F}$ are not invariant like those of $\Gamma_N$, they are strongly returnable \cite{gilbert2002nonlinear} under Assumption~\ref{ass:under_approx}. That is, if $v$ remains constant, the state trajectories are guaranteed to eventually return to $\mc{F}$ due to the properties of the MPC feedback, so the FG simply holds $v$ constant in the meantime. This approach preserves the qualitative theoretical properties (convergence, safety etc.) of the closed-loop system but, unsurprisingly, results in the smaller domain of attraction
\begin{equation}
	\bar{\mc{D}}_{FG} = \bigcup_{v\in \mc{V}_\epsilon} S_x(\mc{F},v) \subseteq \mc{D}_{FG}.
\end{equation}
This smaller domain of attraction is still however large enough to guarantee safe transitions between any $r_1,r_2\in \mc{R}_\epsilon$.

An obvious way to generate the under-approximation is to pick $\mc{F} = \Gamma_i$ for $0\leq i < N$ with the limit case $\mc{F} = \Gamma_0 = \mc{T}$. The ability to use under-approximations also provides the flexibility to design $\mc{F}$ so as to limit the number of inequalities, for example by picking $\mc{F}$ as a box within $\Gamma_N$ or as the convex hull of a pre-specified number of points sampled from the boundary of $\Gamma_N$. This is especially important in embedded applications with memory limitations. Finally, obtaining under-approximations of $\Gamma_N$ through parallelizable approaches, such as sampling based algorithms, is likely key for enabling the application of the FG to higher dimensional systems and an important direction for future work.

\section{Theoretical Analysis} \label{ss:theoretical_properties}
This section analyzes the properties of the closed-loop system under the combined FG and MPC feedback policy. We consider the case from Section~\ref{ss:under-approx} where the under-approximation $\mc{F}$ is used in place of $\Gamma_N$, the results in the nominal case follow by letting $\mc{F} = \Gamma_N$. 

The reference $r$ is assumed constant throughout this section, we suppress any dependencies on $r$ to simplify the notation.

The feasible and invariant sets of the FG are
\begin{gather}
	\Phi \defeq \mc{F} \cap (\reals^{n_x}\times \mc{V}_\epsilon),\\
	\Lambda \defeq \Gamma_N \cap (\reals^{n_x}\times \mc{V}_\epsilon).
\end{gather}
Using these sets, the action of the FG can be expressed as
\begin{subequations} \label{eq:g_def}
\begin{gather} 
	g(x,v) = 
	\begin{cases}
		\bar{g}(x) & (x,v)\in \Phi\\
		v & (x,v)\in \Lambda\setminus\Phi
	\end{cases}\\
	\bar{g}(x) = \argmin_{v\in S_v(\Phi,x)}~\psi(v,r), \label{eq:gbar_def}
\end{gather}
\end{subequations}
where $\psi$ is defined in \eqref{eq:psi_def}. Then the closed-loop dynamics of \eqref{eq:LTI_system} under the combined FG and MPC feedback law are
\begin{subequations} \label{eq:closed-loop-fg-system}
\begin{align} 
v_{k} &= g(x_k,v_{k-1})\\
x_{k+1} &= f(x_k,v_{k}),\\
y_k &= h(x_k,v_k)
\end{align}
\end{subequations}
where $f(x,v) = A x + B\kappa(x,v)$ is as defined in \eqref{eq:closed_loop_dynamics}, $h(x,v) \defeq Cx + D \kappa(x,v)$, and $\kappa$ is the MPC feedback law. The update equations can then be written compactly as
\begin{align}\label{eq:operator-closed-loop-fg}
    (x_{k+1},v_{k+1}) = T(x_k,v_k),
\end{align}
where $T(x, v) \defeq (f(x, v), g(f(x, v), v))$.

The continuity properties of \eqref{eq:closed-loop-fg-system} are as follows.


\begin{lmm} \label{lmm:g_f_continuity}
Given Assumptions~\ref{ass:stabilizable}--\ref{ass:terminal_cost}, the functions $f:\Gamma_N \to \reals^{n_x}$ in \eqref{eq:closed_loop_dynamics} and $\bar g:\Phi \to \mc{V}_\epsilon$ are Lipschitz continuous.
\end{lmm}
\begin{proof}
The MPC feedback policy $\kappa$ and $\bar g$ are solution mappings of strongly convex multi-parametric quadratic programs \eqref{eq:Condensed_LMPC_OCP} and \eqref{eq:gbar_def} and are thus Lipschitz continuous \cite[Theorem 4]{bemporad2002explicit}. Lipschitz continuity of $f$ follows immediately.
\end{proof}

Lipschitz continuity of $\bar g$ is used only when invoking LaSalle's theorem to prove asymptotic stability. As such, the assumption that $\mc{F}$ is polyhedral can be removed, provided continuity of $\bar g$ is maintained. Specifically, if the mapping $S_v(\mc{F},x)$ is continuous in the Pompeiu–Hausdorff sense\footnote{See \cite[Section 3B]{dontchev2009implicit} for a definition.} continuity of $\bar g$ can be proven. 

\subsection{Safety and Recursive Feasibility}
The following theorem provides sufficient conditions under which the (FG) achieves the \textit{Safety} objective and proves that the set $\Lambda$ is forward invariant.

\begin{thm}[Safety \& Invariance]\label{thm:rec_feas}
Given Assumptions~\ref{ass:stabilizable}--\ref{ass:terminal_set}, consider the closed-loop dynamics \eqref{eq:closed-loop-fg-system}. Suppose $x_0 \in \proj{x}\Phi$, then the sequence $\{(x_k,v_k)\}_{k=0}^\infty \subseteq \Lambda$ is well defined and $y_k \in \mc{Y}$ for all $k\in \ints$.
\end{thm}
\begin{proof}
The proof is by induction. At time $k = 0$ if $x_0\in \proj{x}\Phi$ then \eqref{eq:gbar_def} is feasible and $(x_0,v_0)\in \Phi \subseteq \Lambda$. Next, assume $(x_k,v_k)\in \Lambda$, the functions $f$ and $g$ are both defined on $\Lambda$ and thus $(x_{k+1},v_{k+1})$ is well defined. If $(x_{k+1},v_k) \in \Phi$ then $S_v(\Phi,x_{k+1}) \neq \emptyset$, i.e., \eqref{eq:gbar_def} is feasible, and $v_{k+1} = g(x_{k+1},v_k) = \bar{g}(x_{k+1})  \in S_v(\Phi,x_{k+1})$, and thus $(x_{k+1},v_{k+1}) \in \Phi \subseteq \Lambda$. Otherwise, if $(x_{k+1},v_k) \notin \Phi$, \eqref{eq:g_def} yields that $v_{k+1} = v_k$ and thus $(x_{k+1},v_{k+1}) = (x_{k+1},v_k)\in \Lambda$ (by Theorem~\ref{thm:MPC_stab}). Therefore, by induction, $(x_k,v_k)\in \Lambda \subset \Gamma_N$ for all $k \in \ints$ which implies that $\forall k\in \ints$ $y_k \in \mc{Y}$ (by Theorem~\ref{thm:MPC_stab}) and that the sequence $\{(x_k,v_k)\}_{k=0}^\infty$  is well-defined.
\end{proof}

\subsection{Convergence and Stability}
Having established safety, we now consider convergence and stability. We begin by introducing the Lyapunov function candidate
\begin{equation} \label{eq:lyapunov_injective}
	V(v) \defeq \psi(v,r) \geq 0,
\end{equation}
with $\psi$ defined in \eqref{eq:psi_def} and the notation $V_k = V(v_k)$ and $V^* = V(v^*)$ where $v^*$ is defined in \eqref{eq:vstar}. The following Lemma addresses how $V$ evolves along solutions of \eqref{eq:closed-loop-fg-system}.

\begin{lmm} \label{lmm:lyapunov_decrease_estimate}
Given Assumptions~\ref{ass:stabilizable}--\ref{ass:terminal_set}, define the increment
\begin{equation}
	\Delta V(x,v) = V(g(f(x,v),v)) - V(v),
\end{equation}
then for all $(x,v)\in \Lambda$, there exists $\eta > 0$ such that
\begin{equation} \label{eq:lyap_dec_1}
	\Delta V(x,v) \leq -\eta \|g(f(x,v),v)-v\|^2.
\end{equation}
\end{lmm}
\begin{proof}
Partition the set $\Lambda$ into $\Lambda = \Lambda_1 \cup \Lambda_2$ where $\Lambda_1 = \{(x,v)~|~(f(x,v),v) \in \Lambda\setminus \Phi\}$ and $\Lambda_2 = \{(x,v)~|~(f(x,v),v) \in \Phi\}$. If $(x,v)\in \Lambda_1$ then $g(f(x,v),v) = v$ by \eqref{eq:g_def} and \eqref{eq:lyap_dec_1} clearly holds.

Next the case $(x,v)\in \Lambda_2$. Recall that $V$ is a strongly convex quadratic function. Thus, there exists $\eta > 0$ such that
\begin{equation} \label{eq:strong_convex}
	V(v) \geq V(v') + \nabla V(v')^T (v - v') + \eta\|v - v'\|^2
\end{equation}
for all $v',v \in \reals^{n_v}$. Letting $x^+ = f(x,v)$, we have that by \eqref{eq:g_def}, $g(x^+,v) = \bar g(x^+)$ for all $(x,v)\in \Lambda_2$. Moreover, recall that optimality conditions associated with $\bar g(x^+) = \argmin_{s\in S_v(\Phi,x^+)}~V(s)$ are \cite{dontchev2009implicit}
\begin{equation} \label{eq:opt_conditions}
	\nabla V(\bar g(x^+))^T(v - \bar g(x^+)) \geq 0,~~\forall v \in S_v(\Phi,x^+).
\end{equation}
Substituting $v' = \bar g(x^+)$ and \eqref{eq:opt_conditions} into \eqref{eq:strong_convex}, and rearranging, we obtain that, for all $(x,v)\in \Lambda_2$
\begin{equation*}
 	V(\bar g(f(x,v))) - V(v) \leq - \eta\|\bar g(f(x,v)) - v\|^2 \leq 0.
\end{equation*} 
Since $\Lambda = \Lambda_1\cup \Lambda_2$ this completes the proof.
\end{proof}
An immediate consequence is that $\{V_k\}$ is non-increasing.
\begin{cor} \label{lmm:lyapunov_decrement}
Consider \eqref{eq:closed-loop-fg-system}, under Assumptions~\ref{ass:stabilizable}--\ref{ass:under_approx}, if $x_0 \in \Pi_x \Phi$ then
\begin{equation} \label{eq:lyapunov_decrement}
	V(v_{k+1}) - V(v_k) \leq 0.
\end{equation}
\end{cor}





The next Lemma provides a sufficient condition under which the auxiliary reference changes.
\begin{lmm}\label{lmm:convergence}
Given Assumptions \ref{ass:stabilizable}--\ref{ass:under_approx}, define
\begin{equation}\label{eq:ss_tube}
    \mc{B}_\delta(\Sigma) \defeq \{(x,v)~|~v\in \mc{V}_\epsilon,~\|x- G_x v\| \leq\delta\},
\end{equation}
where $\Sigma = \mc{B}_0(\Sigma) = \{(x,v)~|~x = G_x v,~v\in \mc{V}_\epsilon\}$. Then, there exists $\delta^\star>0$ such that $\mc{B}_{\delta^\star}(\Sigma) \subset \Int \mc{F}$. Moreover, $\delta\in [0,\delta^\star]$, $(x,v)\in\mc{B}_{\delta}(\Sigma)$, and $v\neq v^*$ implies that $g(x,v) \neq v$.
\end{lmm}

\begin{proof}
To show that $(x,v)\in \mc{B}_\delta(\Sigma) \land v\neq v^* \implies g(x,v) \neq v$ we will construct a point $v' \in S_v(\Phi,x)$ such that $V(v') < V(v)$. By Assumption~\ref{ass:under_approx}, $\Sigma\subset \Int \mc{F}$ and thus there exists $\delta^\star>0$ such that $\mc{B}_{\delta}(\Sigma)\subset \Int \mc{F}$ for all $\delta \in [0,\delta^*]$. Moreover, because $\mc{B}_{\delta}(\Sigma)\subset \Int\mc{F}$, for any $(x,v)\in \mc{B}_{\delta}(\Sigma)$, there exists $\alpha = \alpha(\delta) > 0$ such that $\mc{B}_\alpha(v) \subseteq S_v(\mc{F},x)$.

Fix any $\delta \in [0,\delta^*]$ and the corresponding $\alpha = \alpha(\delta)$. Then, define the set $\mc{C}_\alpha = \mc{V}_\epsilon \cap \mc{B}_\alpha(v)$, the ray
\begin{equation}
	v'(t) = v + t(v^* - v) \quad t\in [0,1],
\end{equation}
and assume $v\neq v^*$. The first step is to show that $t\in [0,\gamma] \implies v'(t)\in \mc{C}_\alpha$ where $\gamma = \min\left(1,\frac{\alpha}{\|v-v^*\|}\right)\in (0,1]$. To prove this, recall that $\mc{V}_\epsilon$ is convex and $v,v^* \in \mc{V}_\epsilon$ thus $v'(t) \in \mc{V}_\epsilon$ for $t\in [0,1]$. Moreover, $\|v'(\gamma) - v\| \leq \|v'(\alpha/\|v-v^*\|) - v)\| = \alpha$ and therefore $\forall t\in [0,\gamma],~v'(t)\in \mc{C}_\alpha$.

To establish that $V$ decreases along $v'(t)$, recall that $V$ is convex and therefore
\begin{align*}
	V(v'(t)) &= V((1-t)v + t v^*)\\
	         & \leq V(v) - t[V(v) - V^*]
\end{align*}
for all $v\in \mc{V}_\epsilon \setminus v^*$ and $t\in [0,1]$. Further, using that $V^* < V(v)$ for all $v\in \mc{V}_\epsilon \setminus v^*$ and that $\gamma \in (0,1]$ we conclude that
\begin{equation}  \label{eq:lmm31}
	V(v'(\gamma)) < V(v).
\end{equation}
Thus we have constructed a point $v'(\gamma)\in \mc{C}_\alpha \subseteq S_v(\Phi,x)$ satisfying $V(v'(\gamma)) < V(v)$, this implies that
\begin{equation}
	V(g(x,v)) = \min_{s\in S_v(\Phi,x)}~V(s) \leq V(v'(\gamma)) < V(v).
\end{equation}
Finally, strong convexity of $V$ combined with $V(g(x,v)) < V(v)$ implies that $g(x,v) \neq v$ as claimed.
\end{proof}


The next lemma extends Theorem~\ref{thm:MPC_stab} to the case where $v$ is changing.
\begin{lmm} \label{lmm:ISS}
Given Assumptions~\ref{ass:stabilizable}--\ref{ass:terminal_set}, and the system $x_{k+1} = f(x_k,v_k)$, the error signal $e_k = x_k - G_x v_k$ is input-to-state stable (ISS) \cite{jiang2001input} with respect to the input $\Delta v_k = v_{k+1} - v_k$, i.e., there exist $\beta\in \mc{KL}$ and $\gamma \in \mc{K}$ such that
\begin{equation*} \label{eq:ISS_equation}
	\|x_k - G_xv_k\|_Q \leq \beta(k,\|x_0 - G_x v_0\|) + \gamma\left(\sup_{j\geq 0}~\|\Delta v_j\|\right).
\end{equation*}
Moreover, $\gamma$ is an asymptotic gain, i.e.,
\begin{equation} \label{eq:AG_equation}
	\limsup_{k\to\infty}\|x_k - G_x v_k\|_Q \leq \gamma\left( \limsup_{k\to\infty} \|\Delta v_k\| \right).
\end{equation}
\end{lmm}

\begin{proof}
Under Assumptions~\ref{ass:stabilizable}--\ref{ass:terminal_set}, it is well known, see e.g., \cite{mayne2000mpc}, that the optimal cost function of the MPC feedback law, i.e., \eqref{eq:ocp_cost} evaluated at the optimal solution $\mu^*(x,v)$ which we denote by $J:\Gamma_N \to \reals$, is a Lyapunov function for the closed-loop system, i.e., there exist $\alpha,\alpha_l,\alpha_u\in \mc{K}$ such that 
\begin{gather}
	J(f(x,v),v) - J(x,v) \leq -\alpha(\|x - G_x v\|_Q),\\
	\alpha_l(\|x - G_x v\|_Q) \leq J(x,v) \leq \alpha_u(\|x-G_x v\|_Q)
\end{gather}
for all $(x,v)\in \Gamma_N$. Moreover, under our assumptions, $J$ is uniformly continuous\cite[Prop 1]{limon2009input} and thus there exists $\sigma_x, \sigma_v \in \mc{K}_\infty$ such that $|J(x',v') - J(x,v)| \leq \sigma_x(\|x'-x\|) + \sigma_v(\|v'-v\|)$. Hence, for any $(x,v)\in\Lambda$, $v^+\in S_x(\Lambda,x)$ and $x^+ = f(x,v)$, let $\Delta J = J(x^+,v^+) - J(x,v)$ then
\begin{align}
	\Delta J &= J(x^+,v) - J(x,v) + J(x^+,v^+) - J(x^+,v)\\
	& \leq -\alpha(\|x - G_x v\|_Q) + |J(x^+,v^+) - J(x^+,v)|\\
	& \leq -\alpha(\|x - G_x v\|_Q) + \sigma_v(\|v^+ -v\|)
\end{align}
which demonstrates ISS of $e$ with respect to $\Delta v = v^+ - v$\cite[Lemma 3.5]{jiang2001input}. The existence of the asymptotic gain follows immediately from \cite[Lemma 3.8]{jiang2001input}.
\end{proof}


\begin{cor} \label{cor:continuity_fg}
Let Assumptions~\ref{ass:stabilizable}--\ref{ass:terminal_cost} hold, and let $T:\Lambda\to\Lambda$ be the operator defined in \eqref{eq:operator-closed-loop-fg}. Then $\tilde{T}\!: \Phi\to\Lambda$, the restriction of $T$ to $\Phi$, is continuous and can be expressed explicitly as $\tilde T(x, v) \defeq (f(x,v), \bar g(f(x,v)))$.
\end{cor}

\begin{proof}
Recall that, by \eqref{eq:g_def}, for all $(x,v)\in \Phi$ 
\begin{align*}
    T(x,v)  = (f(x,v), \bar g(f(x,v))) = \tilde T(x,v).
\end{align*}
Since $f$ and $\bar g$ are continuous by virtue of Lemma~\ref{lmm:g_f_continuity}, $\tilde T$ is also continuous.
\end{proof}

Having assembled the required components, we are ready to show asymptotic stability.
\begin{thm}[Asymptotic Stability] \label{thm:asymptotic_convergence}
Let Assumptions~\ref{ass:stabilizable}--\ref{ass:under_approx} hold. The point $(x^*, v^*)$, where $x^* = G_x v^*$, is an asymptotically stable equilibrium of \eqref{eq:closed-loop-fg-system} and $x_0 \in \Pi_x \Phi \implies \lim_{k\to\infty}~(x_k,v_k) = (x^*,v^*)$.
\end{thm}
\begin{proof}
First, note that, by Theorem~\ref{thm:rec_feas}, $x_0 \in \Pi_x \Phi$ guarantees that the sequence $\{(x_k,v_k)\}_{k = 0}^\infty \subseteq \Lambda$ is well defined. 
Moreover, the sequence $\{V_k\}_{k=0}^\infty$ is non-increasing (Corollary~\ref{lmm:lyapunov_decrement}) and bounded from below, hence converging. By virtue of Lemma~\ref{lmm:lyapunov_decrease_estimate}, we have that there exists $\eta > 0$ such that
\begin{align}
    \| v_{k+1} - v_k  \|^2 \leq \eta^{-1} | V_{k+1} - V_k | \to 0
\end{align}
as $k\to \infty$ and thus $\lim_{k\to\infty} ||\Delta v_k|| = 0$. Moreover, using Lemma~\ref{lmm:ISS}, there exists $\gamma\in \mc{K}$ such that
\begin{equation} 
	\limsup_{k\to\infty}\|x_k - G_x v_k\|_Q \leq \gamma\left( \limsup_{k\to\infty} \|\Delta v_k\| \right),
\end{equation}
together with the observability of $(A,Q)$, this implies that
\begin{equation} \label{eq:AG_2}
	\lim_{k\to\infty} \|x_k - G_x v_k\| = 0.
\end{equation}
Therefore, there exists $t \geq 0$ such that $\|x_k - G_x v_k\| \leq \delta^\star$ for all $k\geq t$ and thus $(x_k,v_k)\in \mc{B}_{\delta^\star}(\Sigma)$ for all $t\geq k$, where $\delta^\star$ and $\mc{B}_\delta(\Sigma)$ are defined in Lemma~\ref{lmm:convergence}.

By virtue of Lemma~\ref{lmm:convergence}, $\mc{B}_{\delta^*}(\Sigma) \subset \Int \mc F$ implying that $\mc{B}_{\delta^*}(\Sigma) \cap (\reals^n \times \mc V_\epsilon) \subset \Int \mc F \cap (\reals^n \times \mc V_\epsilon) \subsetneq \Phi$, and thus $\{(x_{k}, v_{k})\}_{k=t}^\infty  \subseteq \mc{B}_{\delta^*}(\Sigma) \subset \Phi$. Hence, for all $k\geq t$,
\begin{equation}
	(x_{k+1},v_{k+1}) = \tilde T(x_k,v_k),
\end{equation}
where $\tilde T$ is defined in Corollary~\ref{cor:continuity_fg}. As $\tilde T$ is continuous (Corollary~\ref{cor:continuity_fg}), and $V$ is non-increasing along solutions of \eqref{eq:closed-loop-fg-system} (Corollary~\ref{lmm:lyapunov_decrement}), the invariance principle \cite[Theorem 6.3]{laSalle1976discrete} implies that
\begin{align*}
    (x_{k}, v_{k}) \to \mc{M} \text{ as } k\to\infty
\end{align*}
where $\mc{M} \subset \Phi$ denotes the largest invariant subset of 
\begin{equation}
	\Omega = \{(x,v) \in \Phi~|~ V(\bar g(f(x,v)) - V(v) = 0\}.
\end{equation}
Moreover, \eqref{eq:AG_2} implies that $(x_k,v_k) \to \Sigma$ as $k\to \infty$ and thus $(x_k,v_k) \to \mc{M}\cap \Sigma$ as $k\to \infty$.


We claim that $\mc{M} \cap \Sigma = \{(x^*, v^*)\}$; evidently, $(x^*, v^*) \in \mc{M}$ and $(x^*, v^*) \in \Sigma$. Recall that $\mc{M}\subset \Omega\subset \Phi$, thus by Lemma~\ref{lmm:lyapunov_decrease_estimate},
\begin{equation}
	(x,v)\in \mc{M}  \implies \bar g(f(x,v)) = v
\end{equation}
furthermore, by virtue of Theorem~\ref{thm:MPC_stab},
\begin{equation}
	(x,v)\in \Sigma \implies x = f(x,v)
\end{equation}
and thus
\begin{equation}
    \label{eq:necessary-invariance}
     (x,v)\in \mc{M}\cap \Sigma  \implies \bar g(f(x,v)) = \bar g(x) = v.
\end{equation}
Moreover, by Lemma~\ref{lmm:convergence}, for all $(x, v) \in \Sigma = \mc{B}_0(\Sigma)$ we have that $v \neq v^* \implies g(x) \neq v$ and thus 
\begin{align}
    \label{eq:necessary-sigma}
	 (x, v) \in \Sigma \text{ and } \bar g(x) = v \implies v = v^*.
\end{align}
Taking the logical conjunction of right-hand sides of \eqref{eq:necessary-invariance} and \eqref{eq:necessary-sigma} immediately yields the implication
\begin{equation}
	(x,v) \in \mc{M}\cap \Sigma \implies v = v^*,
\end{equation}
and thus $\mc{M}\cap\Sigma = \{(x,v)~|~x = G_x v,~v= v^*\} = \{(x^*,v^*)\}$ as claimed.


Lyapunov stability of $(x^*,v^*)$ follows from Corollary~\ref{lmm:lyapunov_decrement} and the ISS stability of the tracking error. Therefore, the sequence $\{(x_k,v_k)\}_{k=0}^\infty \subseteq \Lambda$ is well defined, $x_0 \in \Pi_x \Phi$ implies that $(x_k,v_k) \to (x^*,v^*)$ as $k\to \infty$ and $(x^*,v^*)$ is a Lyapunov stable equilibrium point of \eqref{eq:closed-loop-fg-system}.
\end{proof}

\begin{thm}[Finite-time Convergence] \label{thm:finite_time_convergence}
Let Assumptions~\ref{ass:stabilizable}--\ref{ass:under_approx} hold and consider the closed-loop system \eqref{eq:closed-loop-fg-system}. Then, for all $x_0\in \proj{x}\Phi$, there exists $t \geq 0$ such that $v_k = v^*$ for all $k \geq t$.
\end{thm}
\begin{proof}
Thanks to Lemma~\ref{lmm:strictly_ss} we know that $(x^*,v^*)\in \Sigma \subset \Int \mc{F}$ and thus $x^* \in \Int S_x(\mc{F},v^*)$ In addition, the definition of $\Phi$ implies that $S_x(\Phi,v) = S_x(\mc{F},v)$ for all $v \in \mc{V}_\epsilon$ and thus $x^* \in \Int S_x(\Phi,v^*)$.

Since $x_k \to x^* \in \Int S_x(\Phi,v^*)$ as $k\to \infty$ (Theorem~\ref{thm:asymptotic_convergence}) there exists a finite $t \geq 0$ such that $x_t \in S_x(\Phi,v^*)$. From the definition \eqref{eq:g_def}, it is evident that $g(x,v) = \bar{g}(x) = v^*$ for all $x \in S_x(\Phi,v^*)$ and thus $v_t = g(x_t,v_{t-1}) = v^*$. Finally, thanks to Theorem~\ref{thm:MPC_stab}, $x_k \in S_x(\Lambda,v^*)$ implies that $x_{k+1} = f(x_k,v_k) \in S_x(\Lambda,v^*)$ and thus we can consider two cases, corresponding to the partition $\Lambda = \Phi \cup (\Lambda\setminus \Phi)$. If $x_k \in S_x(\Phi,v^*)$ then $v_k = g(x_k,v^*) = \bar{g}(x_k) = v^*$, and if $x_k \in S_x(\Lambda\setminus \Phi,v^*)$ then $v_k = g(x_k,v^*) = v^*$ and thus $v_k = v^*$ for all $k \geq t$.
\end{proof}

\section{Numerical Examples} \label{ss:numerical_examples}
\subsection{Double Integrator} \label{ss:double_integrator}
We first consider a double integrator example, which allows us to visualize the geometries of the sets in the paper. The system matrices are 
\begin{gather*}
    A = \begin{bmatrix} 1 & 0.1 \\ 0 & 1 \end{bmatrix}, 
    B = \begin{bmatrix} 0 \\ 0.1 \end{bmatrix}, 
    C = \begin{bmatrix} 1 & 0 \\ 0 & 1 \\ 0 & 0 \end{bmatrix}, 
    D = \begin{bmatrix} 0 \\ 0 \\ 1 \end{bmatrix}, 
\end{gather*}
$E = \begin{bmatrix} 1 & 0 \end{bmatrix}$, and $F = 0$, and the sampling time is $t_s=0.1$. The default constraint set is 
\begin{equation}
	\mc{Y}_1 = [-1,~1] \times [-0.25,~0.25] \times [-0.25,~0.25],
	\label{eq:const1}
\end{equation}
and the MPC parameters are $Q = I$, $R = 1$, and $N = 10$ unless otherwise specified. The initial condition $x_0 = [-1,~0]^T$, and reference $r = 0.75$ are chosen such that $x_0 \notin S_x(\Gamma_{10},G_z^{-1}r)$. For the following figures, the terminal set is $\mathcal{T} = \tilde{O}_{\infty}^{0.01}$ and is computed using the procedure in Appendix A.

Figure~\ref{fig:gamma_vs_T} illustrates the geometries of $\mathcal{T}$ and $\Gamma_{10}$. The terminal set $\mathcal{T}$ is entirely contained in the feasible set, and in both cases $v$ is implicitly bounded by the constraints on $x_1$. The feasible sets form an increasing sequence of sets in $N$, i.e., $\Gamma_N \subseteq \Gamma_{N+1}$ for all $N\geq 0$. This is illustrated in Figure~\ref{fig:DOA_vs_N} which uses a modified constraint set
\begin{equation*}
	\mc{Y}_2 = [-1,~1] \times [-1,~1] \times [-0.05,~0.05]
\end{equation*}
for clarity. The set $\Gamma_N$ appears to be approaching some $\Gamma_\infty \supseteq \Gamma_N$, we hypothesize that this occurs whenever $\mc{Y}$ is compact. 

\begin{figure}[h]
	\centering
	\includegraphics[width = \columnwidth]{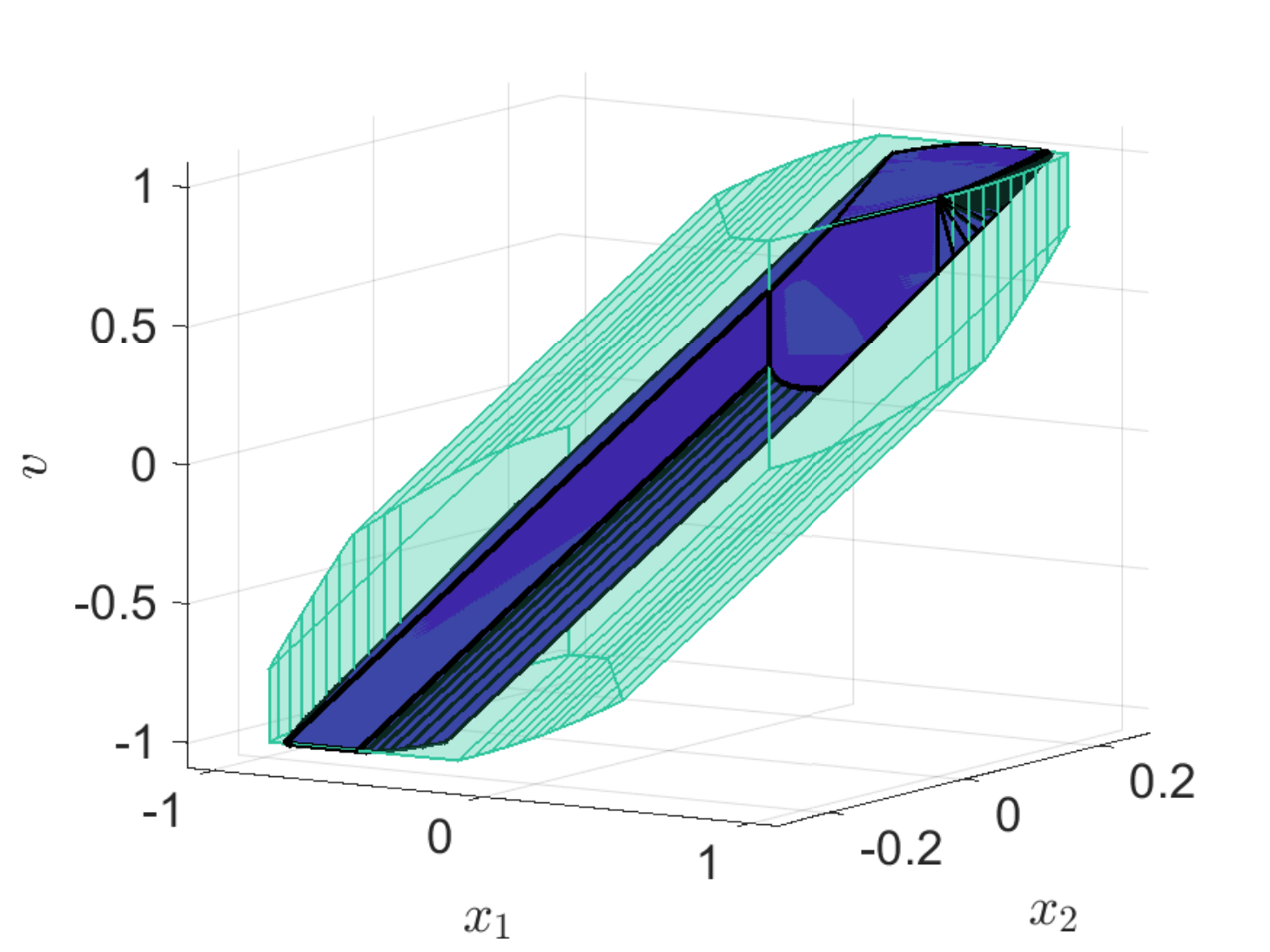}
	\caption{Terminal set $\mathcal{T}$ (blue) encased in feasible set $\Gamma_{10}$ (teal) for the double integrator with constraints $\mathcal{Y}_1$.}
	\label{fig:gamma_vs_T}
\end{figure}

\begin{figure}[h]
	\centering
	\includegraphics[width = \columnwidth]{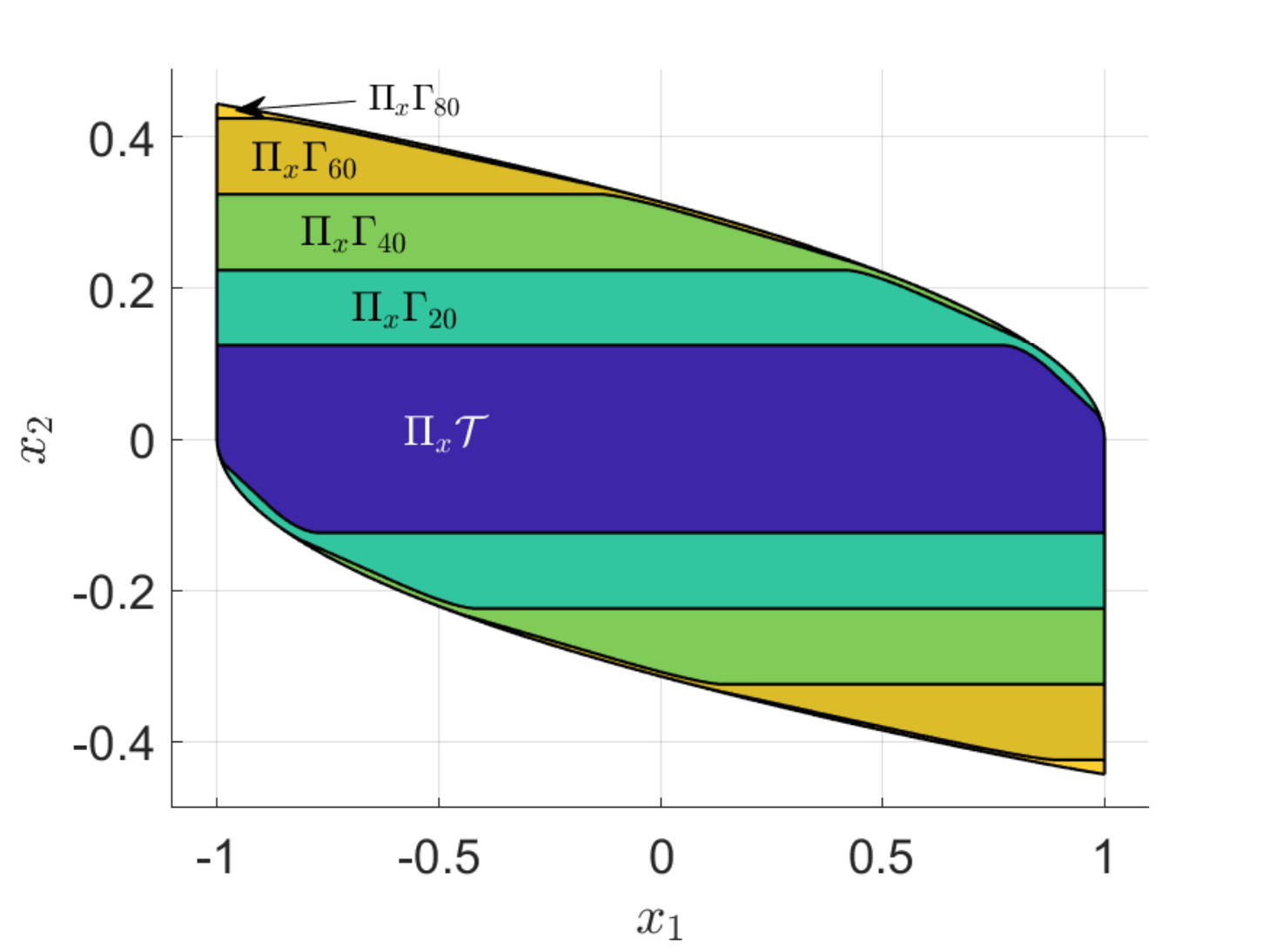}
	\caption{Increasing the control horizon $N$ expands the size of the feasible set while the terminal set stays constant. Here $\mathcal{T} = \Gamma_0 = \tilde{O}_{\infty}^{0.01}$ with constraints $\mathcal{Y}_2$.}
	\label{fig:DOA_vs_N}
\end{figure}

Figure \ref{fig:dint_traj} uses the original constraints \eqref{eq:const1} and displays the responses of the closed-loop system under the MPC + FG feedback policy. All constraints are satisfied and the auxiliary reference converges to $r$ in finite time as predicted by the theory. The same dynamics are displayed in Figure \ref{fig:3D_traj} and illustrates how the MPC + FG navigates $\Gamma_{10}$. By the time $v_k=r$, the current state $x_k$ of the system has entered $\mathcal{D}_{MPC}$ (yellow). From here, the FG holds the auxiliary reference constant and the MPC controller ensures that $x_k \to \bar{x}_r$ as $k\to\infty$.

\begin{figure}[h]
	\centering
	\includegraphics[width = \columnwidth]{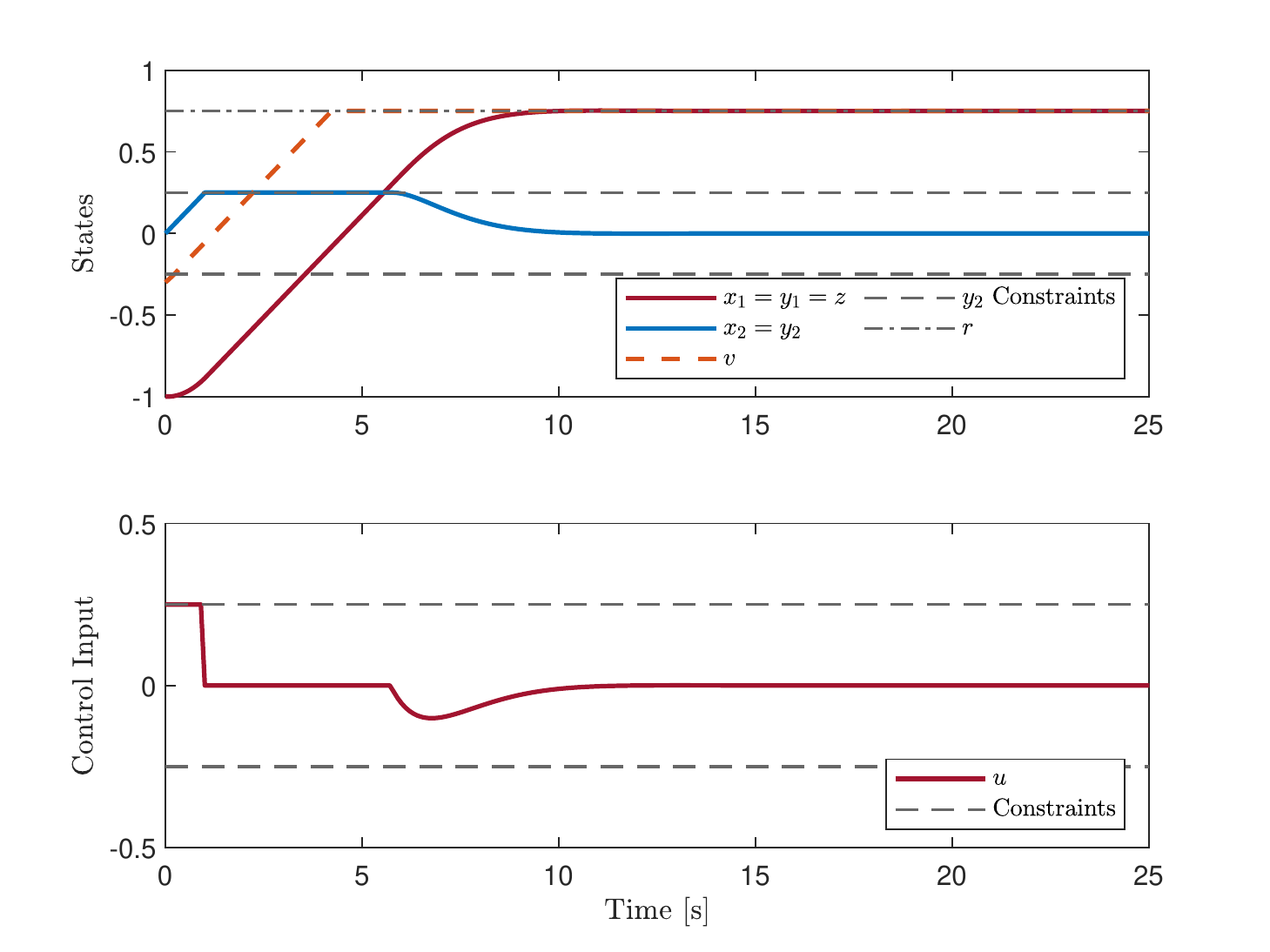}
	\caption{Closed-loop double integrator trajectories for the MPC + FG control law using constraints $\mathcal{Y}_1$.}
	\label{fig:dint_traj}
\end{figure}

\begin{figure}[h]
	\centering
	\includegraphics[width = \columnwidth]{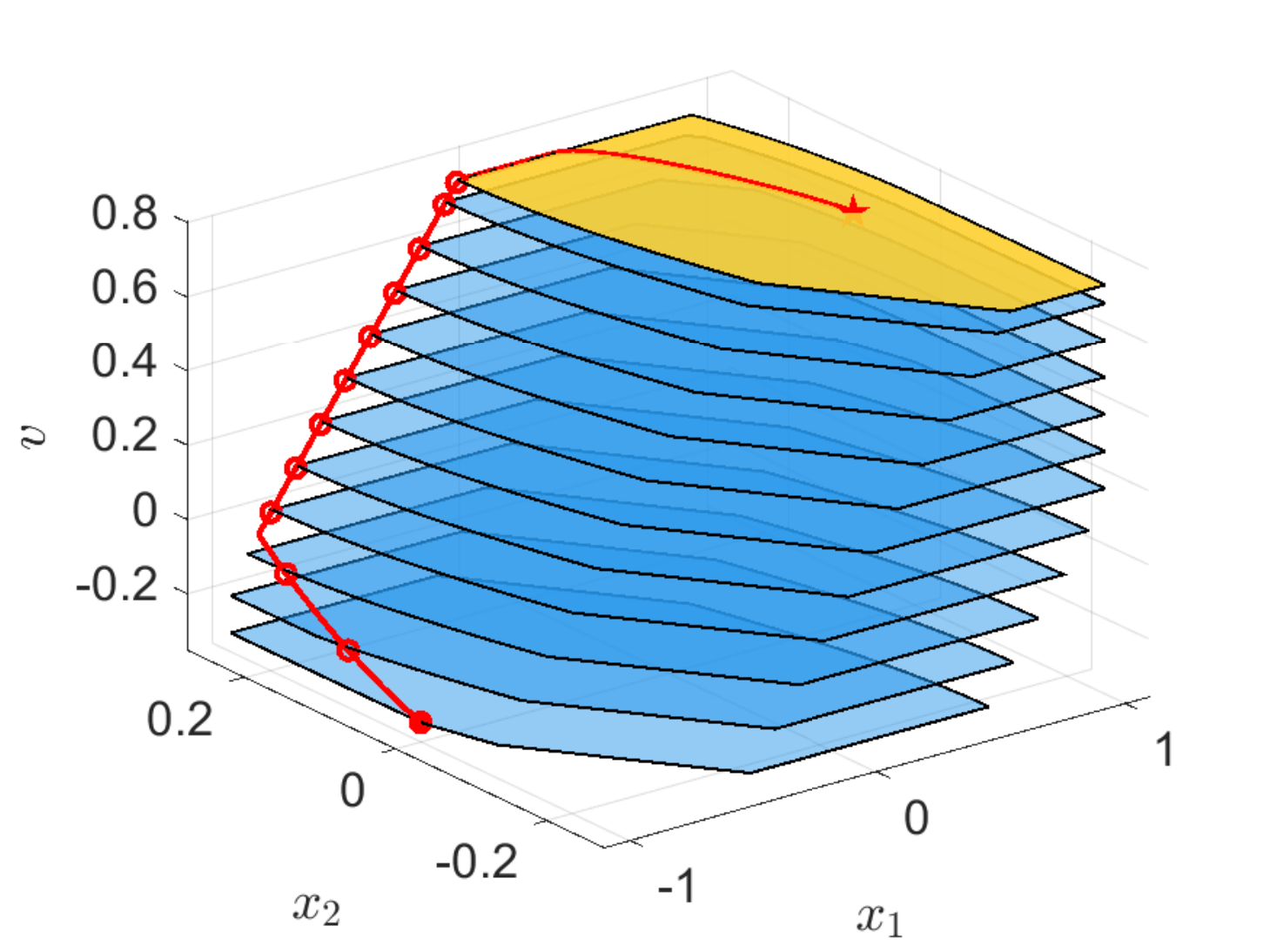}
	\caption{A closed-loop trajectory of the double integrator over the slices $S_x(\Gamma_{10},v)$ for different values of $v$ with constraints $\mathcal{Y}_1$. Circle markers show when the trajectory enters each slice and the star is the point $(x^*_r,v^*_r)$.}
	\label{fig:3D_traj}
\end{figure}

Figure \ref{fig:di_comparison} compares the MPC + FG feedback law with $N = 10$ to an un-goverened MPC controller with $N = N^* = 236$ where
\begin{equation}\label{eq:Nstar}
	N^* = N^*(x_0,r,\mc{T}) = \inf_i~\{i~|~(x_0,G_z^{-1}r)\in \Gamma_i\}
\end{equation}
is the smallest horizon length such that the MPC policy is feasible for the chosen $x_0$. Both these control laws are also compared to a CG combined with an underlying linear quadratic regulator (LQR). All three controllers use $Q=100I$ and $R=1$. The constraint set
\begin{equation*}
	\mc{Y}_3 = [-20,~20] \times [-1,~1] \times [-0.25,~0.25]
\end{equation*}
is chosen to illustrate what happens when the initial conditions $x_0 = [-17,~0]^T$, and reference $r=4$ are chosen far away from each other. As displayed in Figure~\ref{fig:di_comparison}, there is $\approx$ 37\% increase in rise time using the FG vs the ungoverned MPC with $N = 236$, but the worst case computation time for the combined FG and MPC feedback policy, with $N = 10$, is over 5000 times faster than of the ungoverned MPC with $N = 236$, see Table \ref{tab:texe_di}.

\begin{figure}[h]
	\centering
	\includegraphics[width = \columnwidth]{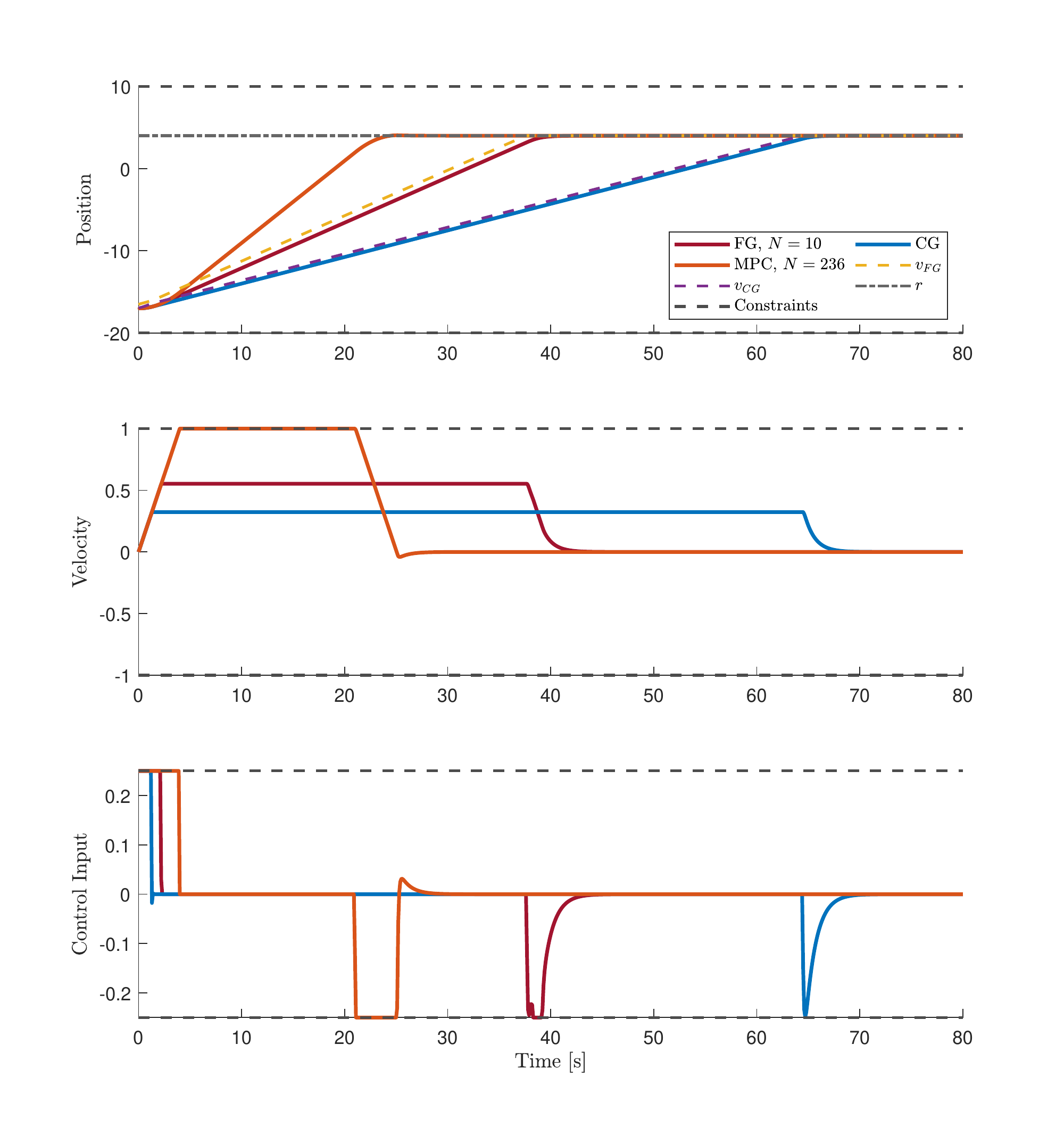}
	\caption{Closed-loop double integrator dynamics for various control laws with constraints $\mathcal{Y}_3$. The FG outperforms the CG, and although the MPC has the best performance, its control horizon is too large for real-time applications.}
	\label{fig:di_comparison}
\end{figure}

\begin{table}[h]
\centering
\caption{Execution time data for the double integrator example.}
\label{tab:texe_di}
\resizebox{\columnwidth}{!}{
\begin{tabular}{|c|c|c|c|c|} \hline
 & FG($N = 10$) & MPC($N =10$) & MPC ($N = 236$) & CG \\\hline
 TAVE [ms] & $0.0126$ & $0.0884$ & $255$ & $0.00833$ \\ \hline
 TMAX [ms] & $0.063$ & $0.345$ & $2170$ & $0.0369$ \\ \hline
\end{tabular}}
\end{table}

\subsection{Lateral Vehicle Dynamics} \label{ss:lateral_vehicle}
This section applies the FG to the lateral dynamics of a car moving forward at a constant longitudinal speed of $V_x = 30 m/s$. The model is based on the one in \cite{wurts2018collision} and roughly represents a 2017 BMW 740i sedan.

\begin{figure}[htbp]
	\centering
	\includegraphics[width=0.95\columnwidth]{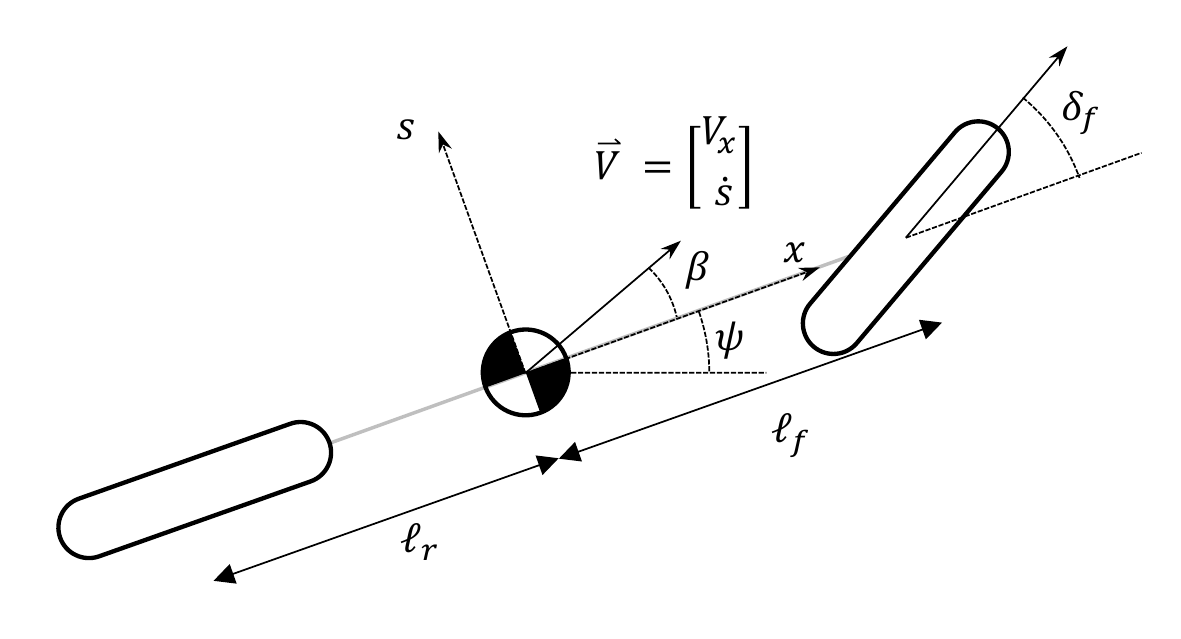}
	\caption{The bicycle model of the lateral vehicle dynamics.}
	\label{fig:bicycle_diagram}
\end{figure}

A diagram of the bicycle model is displayed in Figure~\ref{fig:bicycle_diagram}. The state of the system is $x^T = [s~~ \psi~~ \beta~~ \omega]$ where $s$ is the lateral position of the vehicle, $\psi$ is the yaw angle, $\beta = \dot{s}/V_x$ is the sideslip angle, and $\omega = \dot{\psi}$ is the yaw rate. The control input is the front steering angle $u = \delta_f$ and the system is subject to constraints on $y^T = [\alpha_f ~~ \alpha_r ~~ \delta_f]$ where $\alpha_f$ and $\alpha_r$ are the front and rear slip angles. The tracking output is $z = s$. The system matrices are
\begin{gather*}
	A = \begin{bmatrix}
		0 & V_x & V_x & 0\\
		0 & 0 & 0 & 1\\
		0 & 0 & -\frac{2C_\alpha}{m V_x} & \frac{C_\alpha (\ell_r - \ell_f)}{m V_x^2} - 1\\
		0 & 0 & \frac{C_\alpha (\ell_r - \ell_f)}{I_{zz}} & -\frac{C_\alpha(\ell_r^2 + \ell_f^2)}{I_{zz} V_x} 
	\end{bmatrix},B = \begin{bmatrix}
		0\\ 
		0\\ 
		\frac{C_\alpha}{m V_x}\\ 
		\frac{C_\alpha \ell_f}{I_{zz}}
	\end{bmatrix}, \\
	  C = \begin{bmatrix}
		0 & 0 & -1 & -\frac{\ell_f}{V_x}\\
		0 & 0 & -1 & \frac{\ell_r}{V_x}\\
		0 & 0 & 0 & 0\\
	\end{bmatrix}, ~D = \begin{bmatrix}
		1\\
		0\\
		1\\
	\end{bmatrix},
\end{gather*}

$E = [1~0~0~0]$, and $F = 0$, where $m = 2041~kg$ is the mass of the vehicle, $I_{zz} = 4964~kg\cdot m^2$ is the moment of inertia about the yaw axis, $\ell_f = 1.56~m$ and $\ell_r = 1.64~m$ are the moment arms of the front and rear wheels relative to the center of mass, and $C_\alpha = 246994~N/rad$ is the tire stiffness. The continuous time system matrices are converted to discrete-time using a zero-order hold (\texttt{c2d} in MATLAB) with a sampling time of $t_s = 0.01$ seconds. The constraint set is
\begin{equation}
	\mc{Y} = [-8^\circ,~8^\circ] \times [-8^\circ,~8^\circ] \times [-30^\circ,~30^\circ], 
\end{equation}
which represents limits on the front and rear slip angles (to prevent tire slip and drifting) and a mechanical limit on the steering angle. The initial condition is $x_0 = 0$, the target position is $r = 5~m$, and the weighting matrices are $Q = E^TE$ and $R = 0.1$. The terminal penalty and gain are computed using the linear quadratic regulator and the terminal set is $\mc{T} = \tilde O_\infty^{0.01}$, computed using the procedure in Appendix A.

Figure~\ref{fig:fg_vs_mpc_car} compares the combined FG + MPC feedback law for $N = 15$ with an ungoverned MPC controller with $N = N^* = 76$ where $N^*$ is as in \eqref{eq:Nstar}. The rise and settling times of the combined feedback law is comparable with that of the ungoverned MPC controller despite a $94\%$ reduction in worst-case computation time, see Table~\ref{tab:texe_lv}.
\begin{figure}[htbp]
	\centering
	\includegraphics[width=0.95\columnwidth]{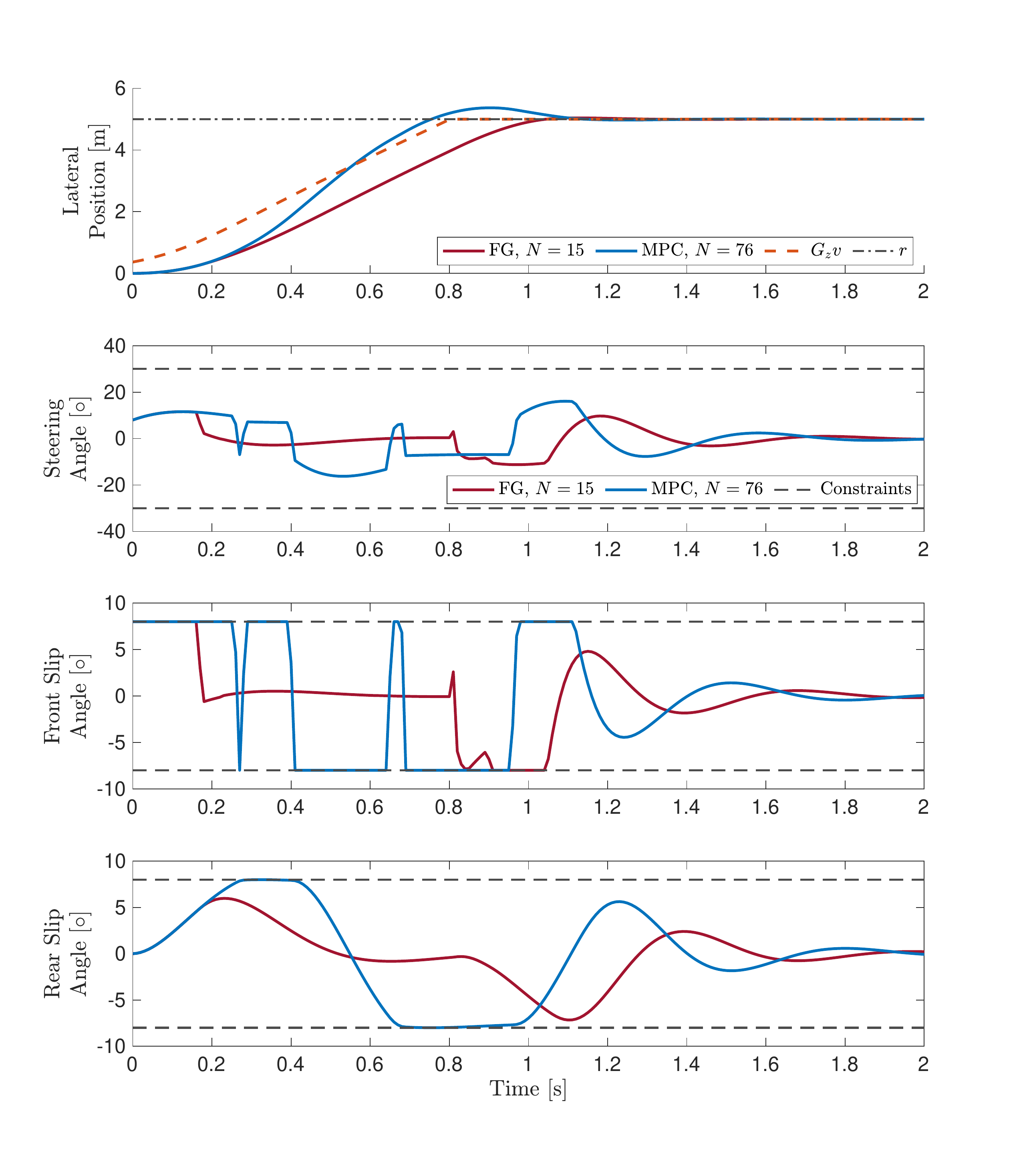}
	\caption{Closed-loop lateral vehicle dynamics responses for the FG with $N=15$ vs. an ungoverned MPC controller with $N=76$, the shortest $N$ such that the initial problem is feasible. The performance (rise-time) of the FG + MPC combination is only marginally slower than the ungoverned MPC controller which needs a significantly longer horizon to ensure feasibility.}
	\label{fig:fg_vs_mpc_car}
\end{figure}

\begin{table}[h]
\centering
\caption{Execution time data for the lateral vehicle dynamics example.}
\label{tab:texe_lv}
\resizebox{\columnwidth}{!}{
\begin{tabular}{|c|c|c|c|} \hline
 & FG($N = 15$) & MPC($N = 15$) & MPC($N = 75$) \\\hline
 TAVE [ms] & $1.4$ & $0.22$ & $11.7$ \\ \hline
 TMAX [ms] & $2.7$ & $0.53$ & $54.5$\\ \hline
\end{tabular}}
\end{table}

Figure~\ref{fig:fig_vs_N_car} compares the response of the closed-loop system for several values of $N$ and with the CG + LQR. As expected, the FG + MPC solution provides a faster response than the CG + LQR solution and the system response becomes faster as $N$ increases. As $N \to N^*$ the filtering effect diminishes until the response of the pure MPC controller is recovered.

\begin{figure}[htbp]
	\centering
	\includegraphics[width=0.95\columnwidth]{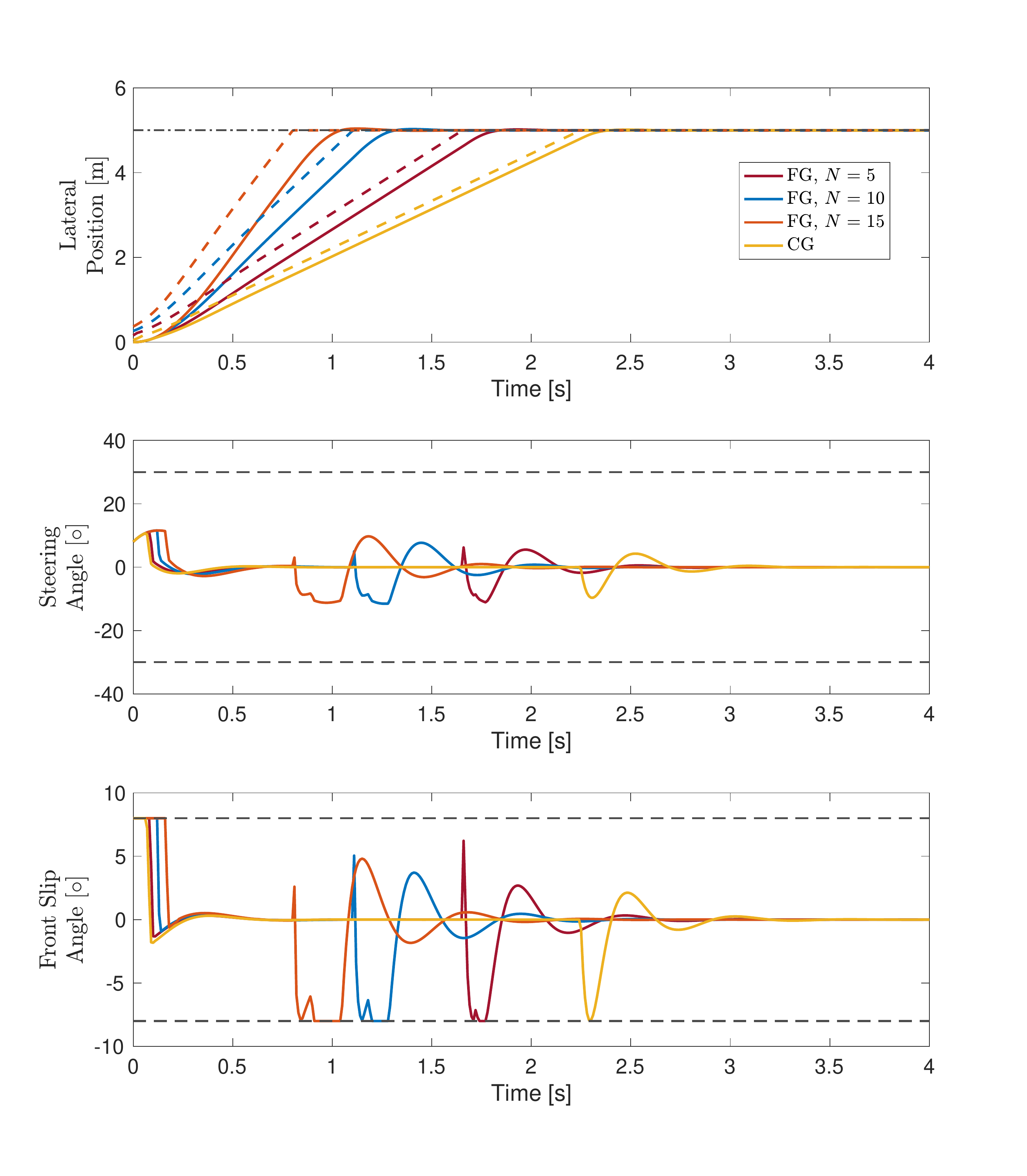}
	\caption{Closed-loop lateral vehicle dynamics responses for varying horizon lengths. The FG outperforms the CG and the system responds more quickly as $N$ increases.}
	\label{fig:fig_vs_N_car}
\end{figure}

\section{Conclusions} \label{ss:conclusions}
This paper has proposed the Feasibility Governor (FG), an add-on unit that expands the region of attraction of linear model predictive controllers by manipulating the reference input passed to the controller and is designed to interfere minimally with the operation of the nominal controller. It was shown that the FG is safe, converges in finite time, and extends the region of attraction of MPC controllers at a fraction of the computation cost associated with increasing the prediction horizon. Future work includes extending the FG to nonlinear settings, and exploring parallelizeable  methods for synthesizing inner-approximation of the feasible set to enable to application of the FG to large scale systems.

\begin{appendix}
\section*{A. Computing the Terminal Set Mapping}
Substituting the terminal control law \eqref{eq:terminal_control_law} into the open-loop dynamics \eqref{eq:LTI_system} and using that $x_v = G_x v$ and $u_v = G_u v$ yields
\begin{align} \label{eq:LTI_system_bar}
x_{k+1} &= \bar{A} x_k + \bar{B} v\\
y_k &= \bar{C} x_k + \bar{D} v \in \mc{Y}
\end{align}
where $\bar{A} = A-BK$, $\bar{B} = B\left(KG_x + G_u\right)$, $\bar{C} = C-DK$, and $\bar{D} = D\left(KG_x + G_u\right)$. This is a standard form in the reference governor literature, see e.g., \cite{gilbert1991linear,gilbert1994nonlinear,garone2017reference}, which makes use of the maximal constraint admissible set,
\begin{multline}
	O_\infty = \{(x,v)~|~\bar{C}\bar{A}^kx + \\
	\bar{C}\left(I-\bar{A}\right)^{-1}\left(I-\bar{A}^k\right)\bar{B}v 
	+ \bar{D}v \in \mc{Y},~\forall k\geq 0\}.
\end{multline}
Since $O_\infty$ is maximal, invariant, and constraint admissible \cite[Theorem 1.1]{gilbert1994nonlinear} $\mc{T}  = O_\infty$ is the largest possible terminal set mapping (for a given terminal feedback law). However, $O_\infty$ might not be representable by a finite number of linear inequalities. Thus whenever $O_\infty$ is not finitely determined, we replace it with 
\begin{equation}
	\tilde{O}_\infty^\epsilon = O_\infty \cap O^\epsilon
\end{equation}
 where $O^\epsilon = \{(x,v)~|~ (\bar{D} + \bar{C}(I-\bar{A})^{-1}\bar{B}) v \in (1-\epsilon) \mc{Y}\}$. The set $\tilde{O}_\infty^\epsilon$ can be made arbitrarily close to $O_\infty$ but is guaranteed to be representable by a finite number of linear inequalities and is still forward invariant and constraint admissible. Algorithms for computing $\tilde{O}_\infty^\epsilon$ are well established and can be found in \cite{gilbert1991linear,kolmanovsky1995maximal}; they yield matrices $T = [T_x~~T_v]$ and a vector $c$, such that
\begin{equation}
	\tilde{O}_\infty^\epsilon = \{(x,v)~|~T_x x + T_v v \leq c\}.
\end{equation}

\section*{B. Condensed Matrix Definitions}
Let $\otimes$ denote the Kronecker product and define
\begin{gather*}
    \hat{A} = \begin{bmatrix}   I \\ A \\ A^2 \\ \vdots \\ A^N \end{bmatrix},
    \hat{B} = \begin{bmatrix}   0 & \cdots & \cdots & 0\\
                                B & 0 & \cdots & 0\\
                                AB & \ddots & \ddots & \vdots\\
                                \vdots & \ddots & \ddots & 0\\
                                A^{N-1}B & \cdots & AB & B \end{bmatrix}\\
    \hat{C} = \begin{bmatrix}   I_N \otimes YC  & 0\\
                                0 & T_x \end{bmatrix}
    \hat{D} = \begin{bmatrix}   I_N \otimes YD \\ 0\end{bmatrix} \\
    \hat{H} = \begin{bmatrix}   I_{N}\otimes Q & 0 \\
                                0 & P \end{bmatrix} \text{ and }
    \hat{T}_v = \begin{bmatrix} 0 \\ T_v \end{bmatrix}
\end{gather*}                
Then the matrices in \eqref{eq:Condensed_LMPC_OCP} are
\begin{gather*}
    H = \hat{B}^T\hat{H}\hat{B} + I_N \otimes R, ~W_x = \hat{B}^T\hat{H}\hat{A},~W = \begin{bmatrix} W_x & W_v \end{bmatrix}\\
    W_v = -\left(W_x G_x + H\left(1_N \otimes G_u \right)\right),\\
    M = \hat{C}\hat{B} + \hat{D},~L = \begin{bmatrix} \hat{C}\hat{A} & \hat{T}_v \end{bmatrix}, \text{ and } b = \begin{bmatrix} 1_N \otimes h \\ c \end{bmatrix},
\end{gather*}
where $1_N$ is a column of $N$ ones.

\section*{C. Proof of Lemma~\ref{lmm:strictly_ss}}
Depending on which condition of Assumption~\ref{ass:xv_in_interior} is satisfied, one of the following holds: 

1) Following from \eqref{eq:terminal_set}, the terminal control law \eqref{eq:terminal_control_law} ensures constraint satisfaction $\forall (x,v)\in\mc{T}$. Therefore, it follows from \eqref{eq:feasible_set} that $\mc{T}\subseteq \Gamma_N$. The statement is then proven by noting $\Sigma\subset\Int\mc{T}\subseteq\Int\Gamma_N$. 

2) Since $(A,B)$ is controllable, there exists a deadbeat gain matrix $L$ such that $(A - BL)^\nu = 0$ \cite{o1981discrete}. 
Thus, given the control law $u_k = \bar u_v-L(x_k \!-\!\bar x_v)$, the closed-loop dynamics of \eqref{eq:LTI_system} 
satisfy $x_k = \bar x_v,~\forall k \geq \nu$. 
Let $O_\infty$ denote the maximum constraint admissible set  \cite{gilbert1991linear} associated with the deadbeat dynamics. It follows by definition that $(x,v)\in O_\infty$ ensures $y_k\in\mathcal{Y}$,
which implies $O_\infty\subseteq\Gamma_N$ due to \eqref{eq:feasible_set}. Since $(A-BL)$ is Schur \cite[Property 2]{o1981discrete} and $v\in\mc{V}_\epsilon$, it follows from \cite[Theorem 2.1]{gilbert1991linear} that $\Sigma\subset\Int O_\infty\subseteq\Int\Gamma_N$. 


\end{appendix}

\begin{IEEEbiography}[{\includegraphics[width=1in,clip,keepaspectratio]{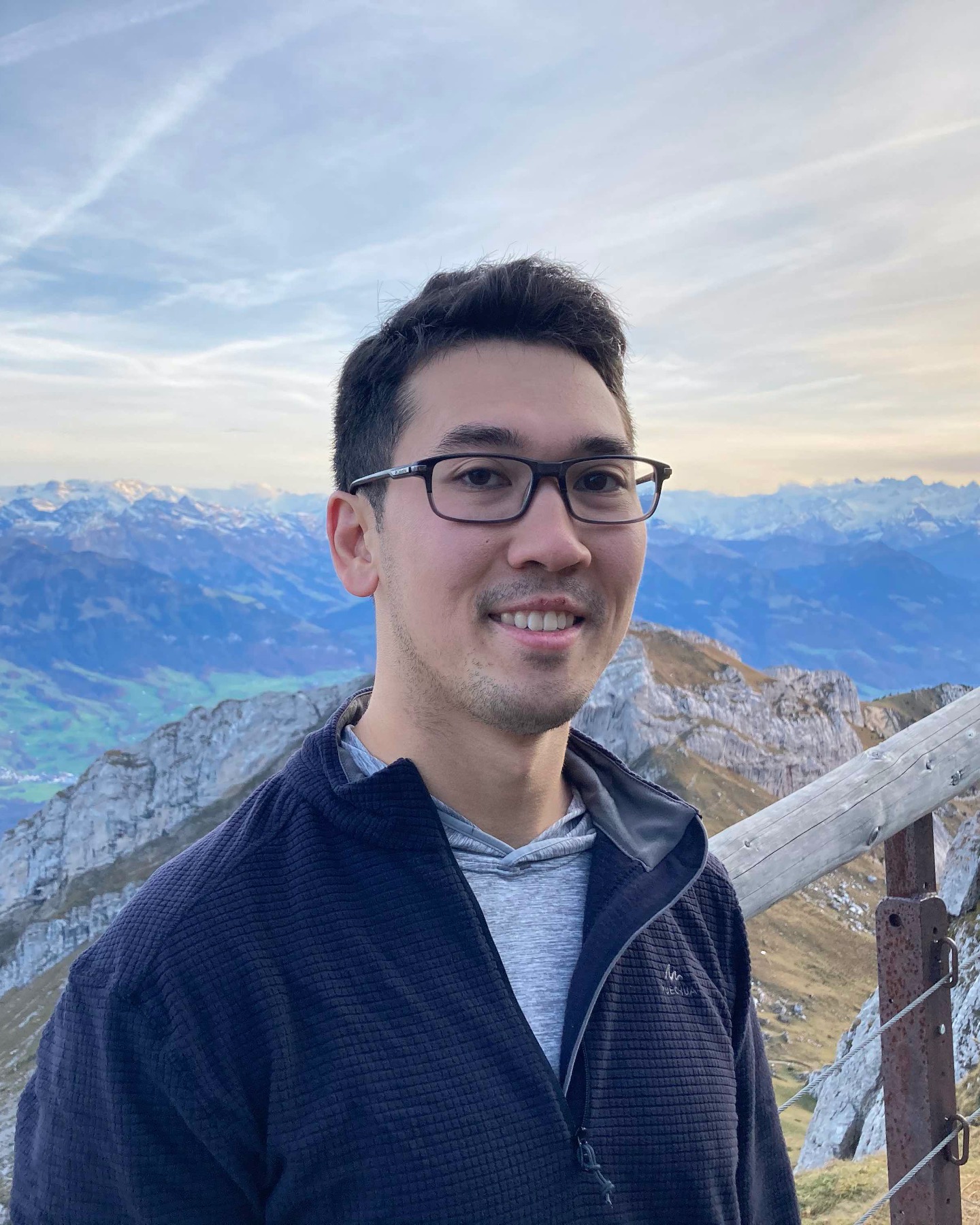}}]{Dominic Liao-McPherson} received his BASc. in Engineering Science from the University of Toronto in 2015 and his PhD. in Aerospace Engineering and Scientific Computing from the University of Michigan  (Ann Arbor) in 2020. He is currently a postdoc in the ETH Zürich Automatic Control Lab. His research interests include constrained control, numerical methods, and algorithms for real-time optimization with applications in aerospace, manufacturing, and energy systems.
\end{IEEEbiography}

\begin{IEEEbiography}[{\includegraphics[width=1in,clip,keepaspectratio]{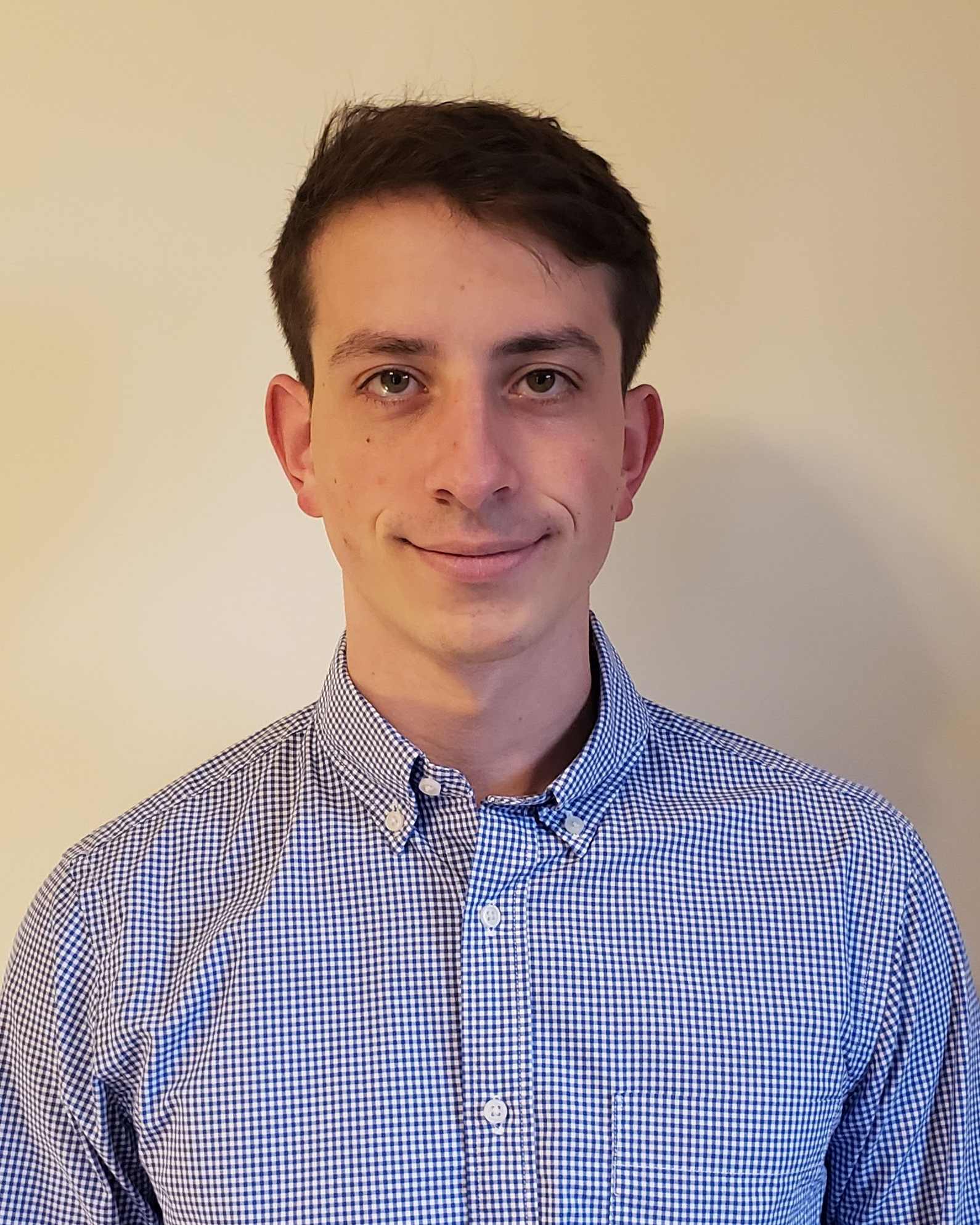}}]{Terrence Skibik} received the B.S degree in electrical engineering from The College of New Jersey, Ewing, NJ USA, in 2019. He is currently working towards the Ph.D. degree at the University of Colorado Boulder, Boulder, CO, USA. His research interests include constrained control and optimization with applications in autonomous systems and renewable energy.
\end{IEEEbiography}

\begin{IEEEbiography}[{\includegraphics[width=1in,clip,keepaspectratio]{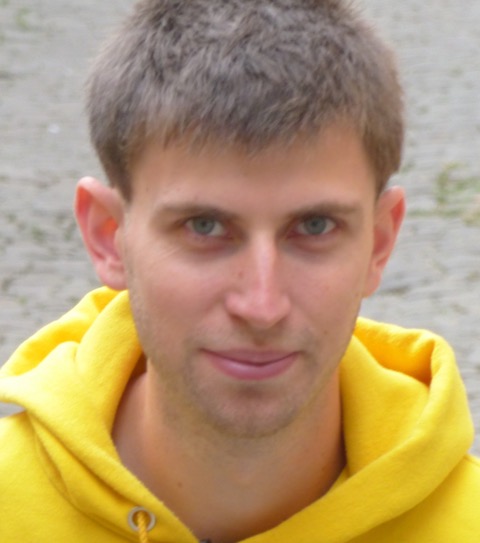}}]{Torbjørn Cunis} is a research fellow of the department of aerospace engineering at the University of Michigan. He received B.Sc. degrees in computer science and aerospace computer engineering from the University of Würzburg in 2013 and 2014, respectively, the M.Sc. degree in automation engineering from the RWTH Aachen University in 2016, and the Dr. degree in systems and control from ISAE-Supaéro, Toulouse, in 2019. His research concerned with the analysis and verification of nonlinear system dynamics, specifically for autonomous vehicles and aircraft, optimal control algorithms, and hybrid system theory. Dr. Cunis is fellow of the Young ZiF at the Centre for Interdisciplinary Research, University of Bielefeld.
\end{IEEEbiography}

\begin{IEEEbiography}[{\includegraphics[width=1in,clip,keepaspectratio]{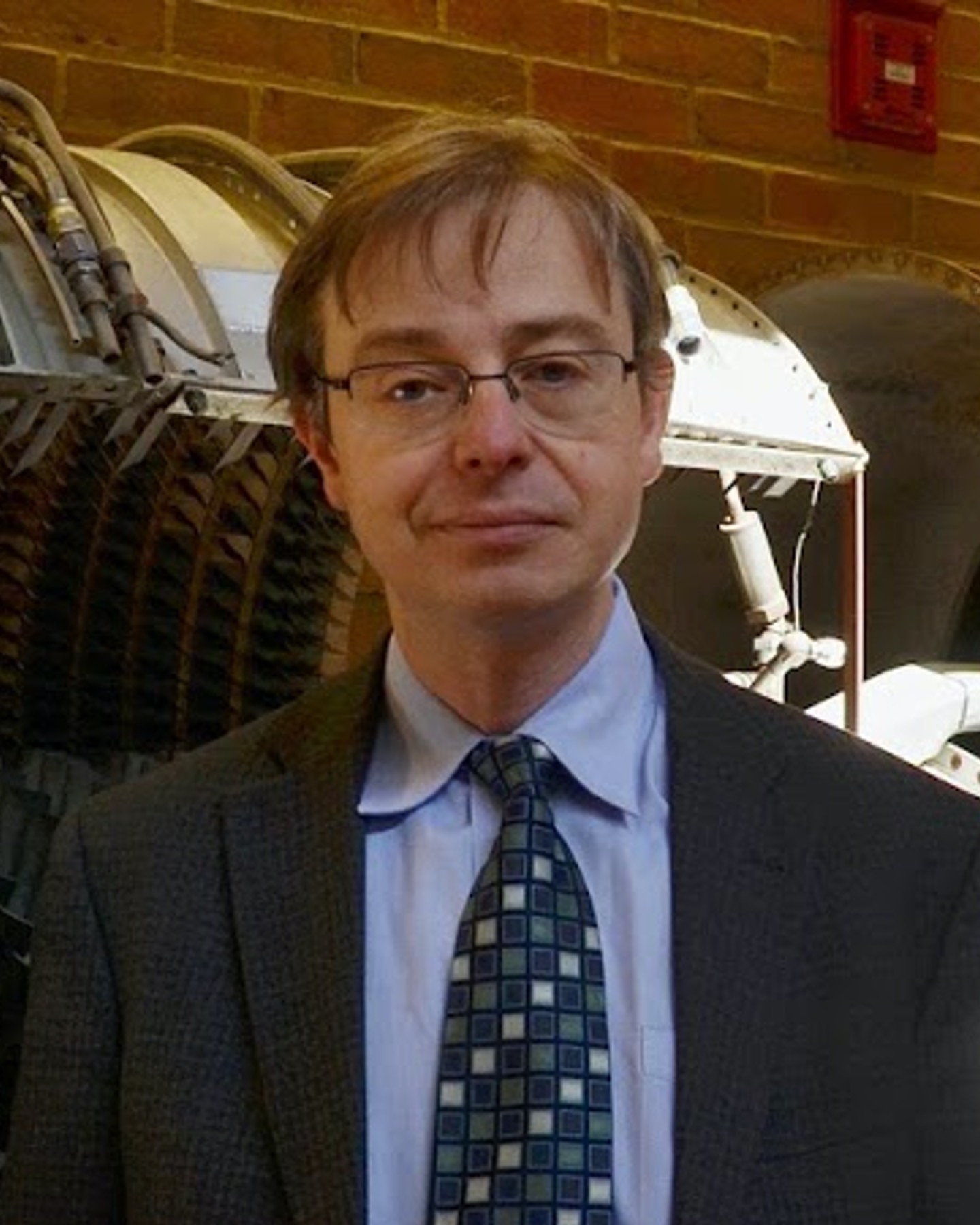}}]{Ilya V. Kolmanovsky} is a professor in the department of aerospace engineering at the University of Michigan, with research interests in control theory for systems with state and control constraints, and in control applications to aerospace and automotive systems.  He received his Ph.D. degree in aerospace engineering from the University of Michigan in 1995. He is a Fellow of IEEE and is named as an inventor on 103 United States patents.
\end{IEEEbiography}

\begin{IEEEbiography}[{\includegraphics[width=1in,clip,keepaspectratio]{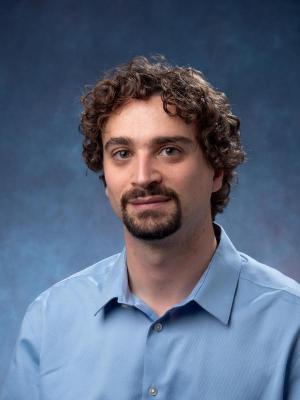}}]{Marco M. Nicotra} received his dual M.S. degree in Mechanical and Electromechanical Engineering from Politecnico di Milano and Universit\'e Libre de Bruxelles, respectively. In 2016 he received his Ph.D. in Systems and Control Engineering from a joint collaboration between Universit\'e Libre de Bruxelles and Universit\`a di Bologna. He was a postdoctoral research fellow at University of Michigan for two years. Since 2018, he is appointed as an Assistant Professor in the department of Electrical, Computer, and Energy Engineering at the University of Colorado, Boulder. His research interests focus in nonlinear and constrained control strategies and their applications to aerospace, energy, robotic, and quantum systems.
\end{IEEEbiography}

\end{document}